\documentclass[12pt]{amsart}
\usepackage{a4wide,enumerate,xcolor}
\usepackage{amsmath,graphicx,comment,enumerate}
\usepackage{mathtools} 
\usepackage{esint} 
\allowdisplaybreaks

\usepackage{enumitem}
\setlist[itemize]{label={$\bullet$}, leftmargin=35pt, itemsep=3pt}

\let\pa\partial
\let\na\nabla
\let\eps\varepsilon
\newcommand{\N}{{\mathbb N}}
\newcommand{\R}{{\mathbb R}}
\newcommand{\diver}{\operatorname{div}}

\newtheorem{theorem}{Theorem}
\newtheorem{lemma}[theorem]{Lemma}
\newtheorem{proposition}[theorem]{Proposition}
\newtheorem{remark}[theorem]{Remark}

\begin{document}

\title[Superlinear nonlocal multispecies population systems]{Superlinear nonlocal diffusion systems \\ 
for multispecies populations: \\
well-posedness and discrete chain rules}


\author[P. Hirvonen]{Peter Hirvonen}
\address{Institute of Analysis and Scientific Computing, TU Wien, Wiedner Hauptstra\ss e 8--10, 1040 Wien, Austria}
\email{peter.hirvonen@tuwien.ac.at}  

\author[A. J\"ungel]{Ansgar J\"ungel}
\address{Institute of Analysis and Scientific Computing, TU Wien, Wiedner Hauptstra\ss e 8--10, 1040 Wien, Austria}
\email{juengel@tuwien.ac.at} 

\author[A. Pollino]{Annamaria Pollino}
\address{Institute of Analysis and Scientific Computing, TU Wien, Wiedner Hauptstra\ss e 8--10, 1040 Wien, Austria}
\email{annamaria.pollino@tuwien.ac.at}

\date{\today}

\thanks{The authors acknowledge partial support from the Austrian Science Fund (FWF), grant 10.55776/PAT2687825, and from the Austrian Federal Ministry for Women, Science and Research and implemented by \"OAD, project MULT09/2025. This work has received funding from the European Research Council (ERC) under the European Union's Horizon 2020 research and innovation programme, ERC Advanced Grant NEUROMORPH, no.~101018153. For open-access purposes, the authors have applied a CC BY public copyright license to any author-accepted manuscript version arising from this submission.} 

\begin{abstract}
A nonlocal diffusion system describing the dynamics of multi-species populations and approximating the Shigesada--Kawasaki--Teramoto (SKT) model is analyzed in a bounded domain. The model consists of a system of integro-differential equations of Andreu--Maz\'on--Rossi--Toledo type, in which a nonlocal cross-diffusion operator acts on nonlinear diffusion potentials. We establish the global existence and uniqueness of strong solutions. The existence proof relies on suitable a priori estimates derived from an entropy inequality, whose proof requires the development of novel discrete chain-rule inequalities. As a by-product, these inequalities provide the basis for the construction of structure-preserving finite-volume schemes for the corresponding local SKT system. Uniqueness of solutions is established by means of a duality argument. Finally, one-dimensional numerical simulations illustrate the nonlocal-to-local limit and the qualitative behavior of solutions for superlinear and sublinear diffusion potentials.
\end{abstract}

 

\maketitle


\section{Introduction}

Nonlocal diffusion models have become an important framework for describing dispersal processes in population dynamics, particularly in situations where individuals are able to perform long-range movements that cannot be adequately represented by classical local diffusion. While a substantial body of literature has focused on nonlocal diffusive differential operators \cite{CHS18,DEF18,JPZ22}, a different class of nonlocal diffusion equations, systematically studied in \cite{AMRT10}, has attracted increasing attention due to its distinctive mathematical structure and its close connection with nonlocal interaction mechanisms. Here, diffusion is governed by integral operators that depend on pairwise differences of the unknown function, instead of second-order derivatives, leading to a nonlinear and genuinely nonlocal evolution. In this paper, we study systems of such nonlocal diffusion equations approximating the Shigesada--Kawasaki--Teramoto (SKT) population model with superlinear pressure functions \cite{CDJ18}. Our main objective is to establish the global existence and uniqueness of strong solutions to these models. The main technical novelty of our approach is the development of new discrete chain-rule inequalities. Besides providing the a priori estimates needed for the existence analysis, they also form the basis for structure-preserving two-point flux approximation finite-volume schemes for SKT systems.

\subsection{Model equations}

The dynamics of the population species with densities $u_1,\ldots,u_n$ is supposed to be governed by the equations
\begin{align}
  \pa_t u_i(t,x) = \int_\Omega J(x-y)\big(P_i(u(t,y))
  - P_i(u(t,x))\big)dy + u_i(t,x)f_i(u(t,x)) \label{1.eq}
\end{align}
for $x\in\Omega$, $t>0$, $i=1,\ldots,n$, and $u=(u_1,\ldots,u_n)$, where $\Omega\subset\R^d$ ($d\ge 1$) is a bounded domain and $J:\R^d\to[0,\infty)$ is the diffusion kernel. The initial conditions equal
\begin{align}\label{1.ic}
  u_i(0,\cdot) = u_i^0\quad\mbox{in }\Omega,\ i=1,\ldots,n,
\end{align}
the diffusion functions are given by 
\begin{align}
  P_i(u) = u_ip_i(u), \quad 
  p_i(u) = a_{i0} + \sum_{j=1}^n a_{ij} u_j^s, \quad i=1,\ldots,n,
  \ s>1, \label{1.p}
\end{align}
and the source terms are of Lotka--Volterra type with
\begin{align}
  f_i(u) = b_{i0} - \sum_{j=1}^n b_{ij}u_j, \quad i=1,\ldots,n,
  \label{1.f}
\end{align}
where $a_{i0}$, $a_{ij}$, $b_{i0}$, $b_{ij}\ge 0$ (competition case). This model was investigated in \cite{GaVe19} in the special case $s=1$ and $n=2$. We generalize the system to the superlinear case $s>1$ and an arbitrary species number $n\in\N$.

Equations \eqref{1.eq} can be interpreted as a nonlocal approximation of the SKT equations
\begin{align}\label{1.skt}
  \pa_t u_i = \Delta P_i(u) + u_if_i(u)\quad\mbox{in }\Omega,
  \quad \na P_i(u)\cdot\nu = 0 \quad\mbox{on }\pa\Omega,\ t>0,
\end{align}
where $\nu$ denotes the exterior unit normal vector to $\pa\Omega$ and $i=1,\ldots,n$. This model was proposed in \cite{SKT79}, derived from stochastic interacting many-particle systems in \cite{CDHJ21}, and mathematically analyzed in \cite{CDJ18}. The diffusion $\Delta P_i(u) = \diver(p_i(u)\na u_i) + \diver(u_i\na p_i(u))$ consists of the sum of the diffusion with coefficient $p_i(u)$ and a contribution due to the mass flux $-u_iv_i$ with the velocity $v_i = -\na p_i(u)$. By Darcy's law, we can interpret $p_i(u)$ as the pressure of the $i$th species. The connection between the nonlocal system \eqref{1.eq} and the local system \eqref{1.skt} is established through the scaling
\begin{align*}
  J_\eps(z) = \frac{K}{\eps^{d}}J\bigg(\frac{z}{\eps}\bigg),
  \quad\mbox{where }K^{-1} = \frac12\int_{\R^d}J(z)z_d^2 dz,
\end{align*}
for normalized kernel functions $J$ and the property \cite{AMRT10}
\begin{align*}
  \frac{1}{\eps^2}\int_\Omega J_\eps(x-y)g(x)dy \to \Delta g(x)
  \quad\mbox{as }\eps\to 0.
\end{align*}
The localization limit was proved for $s=1$ and $n=2$ in \cite{GaVe22} using duality estimates, whose extension to $s\neq 1$ is an open problem. Numerical experiments in one space dimension (see Section \ref{sec.numer}) strongly support the validity of the limit, suggesting that the main obstacles are of a technical nature.

Although nonlocal diffusion equations of the form \eqref{1.eq} generally lack any regularizing effects characteristic of local diffusion operators, they offer several advantages. First, they naturally capture long-range dispersal mechanisms, which are often more realistic in biological applications than purely local diffusion. Second, they admit a stronger well-posedness theory: We establish the existence of bounded strong solutions to \eqref{1.eq}, whereas the available analytical techniques for the corresponding local SKT system \eqref{1.skt} are, in general, limited to proving the existence of weak solutions.

There exist other nonlocal versions of the SKT model. For instance, the authors of \cite{DiMo24} consider convolution-type operators $\Delta(u_i(u_j*J))$ and study their entropy structure. The localization limit in these models was rigorously performed in \cite{Mou20}, assuming a triangular matrix $(a_{ij})$. A related localization limit to a population system without self-diffusion was proved in \cite{BuEs23,DHPP24}. 

\subsection{Key ideas} 

The analysis of the local SKT system \eqref{1.skt} with $s=1$ is based on the fundamental observation that the Boltzmann entropy density $h_B(u)=\sum_{i=1}^n u_i(\log u_i-1)$ yields the crucial a priori estimates required for the existence theory \cite{CDJ18}. Remarkably, the same functional remains effective in the nonlocal setting \eqref{1.eq} with $s=1$, providing analogous estimates \cite{GaVe19}. This is possible under the condition that the coefficient matrix $(a_{ij})_{i,j=1}^n$ is symmetric positive definite. In fact, it is sufficient that $(a_{ij})$ is positively stable and fulfills the detailed-balance condition $\pi_i a_{ij}=\pi_j a_{ji}$ for all $i\neq j$ and some $\pi_i>0$; see \cite{CDJ18}. When $s>1$, the Boltzmann entropy density has to be replaced by the Tsallis entropy density \cite{Tsa88}
\begin{align*}
  h(u) = \sum_{i=1}^n h_i(u_i) = \sum_{i=1}^n\frac{u_i^s-u_i}{s-1}.
\end{align*}
Observe that $h(u)$ converges to the Boltzmann entropy density as $s\to 1$. It was proved in \cite{CDJ18}, still for the local SKT model, that this entropy provides a priori estimates if the detailed-balance condition is satisfied and self-diffusion dominates cross-diffusion in the sense
\begin{align}\label{1.a1}
  Z_i := a_{ii} - \frac{s-1}{s+1}\sum_{j=1,\,j\neq i}^n a_{ij} > 0
  \quad\mbox{for }i=1,\ldots,n.
\end{align}
Indeed, a computation shows that, along smooth solutions to the local system \eqref{1.skt} \cite[Lemma 8]{CDJ18},
\begin{align*}
  \frac{d}{dt}\int_\Omega h(u)dx + c\sum_{i=1}^n a_{i0}
  \int_\Omega |\na u_i^{s/2}|^2 dx 
  + c\sum_{i=1}^n Z_i\int_\Omega |\na u_i^{s}|^2 dx 
  \le C + C\int_\Omega h(u)dx,
\end{align*}
for some constants $c>0$ and $C>0$. By Gronwall's inequality, this yields gradient bounds for $u_i^{s}$ (and $u_i^{s/2}$ if $a_{i0}>0$), essential for the compactness argument. 

Our aim is to derive an analogous entropy inequality for solutions to the nonlocal system \eqref{1.eq}. The main difficulty lies in identifying a suitable discrete counterpart of the chain rule $\na u_i^s = (s/(s-1))u_i\na u_i^{s-1}$. Although this problem has been successfully addressed in the finite-volume setting through the introduction of appropriate mean functions \cite{JuZu21}, this approach does not extend to the present nonlocal model. Our idea is to derive {\em inequalities}, which replace the chain-rule {\em equalities}; see Lemmas \ref{lem.chain1}--\ref{lem.chain2}. We prove in Lemma \ref{lem.aei1} that 
the following entropy inequality holds:
\begin{align*}
  \frac{d}{dt}&\int_\Omega h(u)dx
  + \frac{2}{s}\sum_{i=1}^n a_{i0}\int_\Omega\int_\Omega J(x-y)
  \big(u_i(x)^{s/2}-u_i(y)^{s/2}\big)^2 dydx \\
  &\phantom{xx}+ \frac{s}{8(s-1)C(s)}\sum_{i=1}^n
  \bigg(a_{ii} - C(s)\sum_{j=1, j\neq i}^n a_{ij}\bigg)
  \int_\Omega\int_\Omega J(x-y)\big(u_i(x)^s-u_i(y)^s\big)^2 dydx \\
  &\le C + C\int_\Omega h(u)dx,
\end{align*}
where $C>0$ only depends on the data. After an application of Gronwall's inequality, it provides uniform bounds if $a_{i0}\ge 0$ and the weak cross-diffusion condition
\begin{align}\label{1.a}
  \kappa(s) := \min\bigg\{a_{ii} - C(s)\sum_{j=1, j\neq i}^na_{ij}\bigg\}
  > 0\quad\mbox{for }i=1,\ldots,n,
\end{align}
holds, where
\begin{align}\label{1.Cs}
  C(s) = \begin{cases}
  \displaystyle \frac{s^2}{2(s+1)(s-1)} &\mbox{if }1<s<2, \\[10pt]
  \displaystyle \frac{4(s-1)^2+1}{4(s+1)(s-1)}
  &\mbox{if }s\ge 2.
  \end{cases}
\end{align}
A similar inequality can be proved if instead of \eqref{1.a} the matrix $(a_{ij})$ is positive definite. This yields an $L^s(\Omega)$ bound for $u_i$ (uniform in time), from which we are able to derive an $L^\infty(\Omega)$ estimate and eventually an $W^{1,\infty}(\Omega)$ bound for $u_i$. Our approximation scheme consists of the implicit Euler discretization and a regularization of the kernel. The $W^{1,\infty}(\Omega)$ bound and a bound on the discrete time derivative allow us to apply the Aubin--Lions compactness lemma to conclude the strong convergence of a subsequence of the approximating solutions. 

Condition \eqref{1.Cs} is more restrictive than \eqref{1.a1}. This is a consequence of the discrete chain-rule inequalities employed in our analysis, which are generally not optimal compared with their continuous counterparts (see, e.g., \eqref{2.ineq1} below). It is unlikely that these inequalities can be sharpened to recover the same factors as in the corresponding continuous chain rules.

A noteworthy feature of our analysis is that the constructed solution to \eqref{1.eq}--\eqref{1.ic} is bounded. This is in contrast to the local SKT system \eqref{1.skt}, for which the boundedness of solutions has been established only in certain special cases \cite[Theorem 1]{JuZa16}. Moreover, the nonlocal structure allows us to prove the uniqueness of the strong solution to \eqref{1.eq}--\eqref{1.ic}, while only the weak--strong uniqueness property could be shown for the local SKT system \cite{ChJu19}. 

\subsection{Main results}

We impose the following assumptions:
\begin{itemize}
\item[(A1)] Domain: $\Omega \in \R^d$ ($d\geq 1$) is a bounded domain with Lipschitz boundary.
\item[(A2)] Kernel: $J$ is even, nonnegative, and satisfies $J\in L^\infty(\R^n)\cap BV(\R^n)$ as well as 
\begin{equation}\label{1.meta}
  0<J_*\le \int_\Omega J(x-y)dy \le J^* < \infty
  \quad\mbox{for all }x\in\Omega.
\end{equation}
\item[(A3)] Initial data: $u_i^0\in L^\infty(\Omega)\cap BV(\Omega)$ and $u_i^0\ge 0$ for $=1,\ldots,n$.
\item[(A4)] Data: $a_{i0}$, $a_{ij}$, $b_{i0}$, $b_{ij}\geq 0$ for $i,j=1,\ldots,n$. 
\item[(A5a)] Diffusion matrix: $(a_{ij})_{i,j=1}^n\subset \R^{n\times n}$ is symmetric and satisfies condition \eqref{1.a}.
\item[(A5b)] Diffusion matrix: $(a_{ij})_{i,j=1}^n\subset \R^{n\times n}$ is symmetric positive definite.
\end{itemize}

For given $T>0$, we set $\Omega_T=(0,T)\times\Omega$. 

\begin{theorem}[Global existence]\label{thm.ex}
Let $T>0$ and let Assumptions (A1)--(A4) and either (A5a) or (A5b) hold.
Then there exists a strong solution to \eqref{1.eq}--\eqref{1.ic} satisfying $u_i\ge 0$ in $\Omega_T$,
\begin{align*}
  u_i\in W^{1,\infty}(0,T;L^\infty(\Omega))\cap
  C^0([0,T];L^\infty(\Omega)\cap BV(\Omega)),
\end{align*}
and $u_i$ solves \eqref{1.eq}--\eqref{1.ic} a.e.\ in $\Omega$ and for all $t\in[0,T]$, $i=1,\ldots,n$. Moreover, there exists a constant $C(u^0,T)>0$ such that
\begin{align}
  \int_\Omega& h(u(t))dx
  + \kappa_0\sum_{i=1}^n\int_0^t\int_\Omega\int_\Omega
  J(x-y)\big(u_i(\tau,x)^{s/2}-u_i(\tau,y)^{s/2}\big)^2
  dydxd\tau \nonumber \\
  &+ \kappa_1\sum_{i=1}^n\int_0^t\int_\Omega\int_\Omega
  J(x-y)\big(u_i(\tau,x)^{s}-u_i(\tau,y)^{s}\big)^2
  dydxd\tau \le C(u^0,T). \label{1.ei} 
\end{align}
where $\kappa_0=(2/s)\min_{i=1,\ldots,n}a_{i0}$ and $\kappa_1>0$ depends on $\kappa(s)$, as defined in \eqref{1.a}, if Assumption (A5a) is satisfied, and on $s$ and the smallest eigenvalue of $(a_{ij})$ if Assumption (A5b) holds.
\end{theorem}

The proof of Theorem \ref{thm.ex} is based on an approximation argument, Banach's fixed-point theorem, uniform bounds from the entropy inequality, and the Aubin--Lions compactness lemma. We do not address the more challenging case $s<1$, as it is currently unclear how to derive suitable discrete chain-rule inequalities in this regime. 

\begin{theorem}[Uniqueness of strong solutions]\label{thm.unique}
Let the assumptions of Theorem \ref{thm.ex} hold and let $u$ and $\bar u$ be two strong solutions. Then $u=\bar u$ a.e.\ in $\Omega_T$.
\end{theorem}

The idea of the proof is to use the duality method. More precisely, we first derive the equations satisfied by the difference of two strong solutions $u$ and $\bar u$ and then test it with the solution of an appropriate linear dual problem associated with these solutions, following the approach of \cite[Theorem 5]{GaVe19}. This yields directly $\int_\Omega(u-\bar u)(t)^2dx=0$, from which we conclude that $u(t)=\bar u(t)$ a.e.\ in $\Omega$. 

The long-time behavior of solutions is widely an open problem. Applying the relative entropy relative to the constant steady state $u^\infty$, determined from the initial data, we first need to estimate the entropy production term in case $\kappa_0>0$ from below by the difference $(u_i^{s/2}(x)-u_i^{s/2}(y))^2$ in the nonlocal term. This is possible if and only if $s\ge 2$. Second, by the nonlocal Poincar\'e inequality \cite{AMRT10}, the nonlocal difference is estimated from below by the difference $u_i-u_i^\infty$ in the $L^s(\Omega)$ norm. Third, this term can be bounded from below by the relative entropy if and only if $1<s\le 2$. Thus, we can expect a long-time asymptotics result only in the case $s=2$. We detail the arguments in Remark \ref{rem.time}. Observe that results in the literature are usually concerned with {\em linear} nonlocal diffusion operators, see, e.g., \cite[Theorem 3]{CCR06}. 

The paper is organized as follows. In Section \ref{sec.chain}, we prove the discrete chain rules needed for the analytical results. Theorems \ref{thm.ex} and \ref{thm.unique} are proved in Sections \ref{sec.ex} and \ref{sec.unique}, respectively. Some numerical experiments in one space dimension are presented in Section \ref{sec.numer}. Finally, we show in Section \ref{sec.fvm} that the discrete chain rules developed in this work allow for the construction of entropy-dissipating finite-volume schemes for cross-diffusion systems of the form $\pa_t u_i = \Delta P_i(u)$. 


\section{Discrete chain rules}\label{sec.chain}

We prove some discrete chain rules needed to derive (approximate) entropy inequalities. The following inequality was proved in \cite[Lemma A.3]{CJS16}. 

\begin{lemma}\label{lem.chain1}
Let $a$, $b\ge 0$ and $\beta$, $\gamma>0$. Then
\begin{align*}
  (a^\beta-b^\beta)(a^\gamma-b^\gamma) 
  \ge \frac{4\beta\gamma}{(\beta+\gamma)^2}
  (a^{(\beta+\gamma)/2} - b^{(\beta+\gamma)/2}).
\end{align*}
\end{lemma}

The previous inequality is a discrete counterpart of the continuous chain rule $\na u^\beta\cdot\na u^\gamma = (4\beta\gamma/(\beta+\gamma)^2)|\na u^{(\beta+\gamma)/2}|^2$. The following lemma establishes new chain-rule inequalities, which may be viewed as discrete analogues of the chain rule identity $(1-1/s)\na u^s = u\na u^{s-1}$. 

\begin{lemma}\label{lem.chain2}
Let $a\ge b>0$. Then
\begin{align}
  \bigg(2-\frac{1}{s}\bigg)(a^s-b^s)
  \ge (a^{s-1}-b^{s-1})(a+b) 
  \ge \bigg(1-\frac{1}{s}\bigg)(a^s-b^s) &\quad\mbox{if }s > 1, 
  \label{2.ineq1} \\
  (a^s-b^s)^2 \ge (a^{s-1}-b^{s-1})^2(a+b)^2 
  &\quad\mbox{if }1<s\le 2, \label{2.ineq12} \\
  (a^s-b^s)^2 \ge \frac{s^2}{4(s-1)^2}(a^{s-1}-b^{s-1})^2(a+b)^2 
  &\quad\mbox{if }s\ge 2. \label{2.ineq2}
\end{align}
Inequalities \eqref{2.ineq12} and \eqref{2.ineq2} hold for any $a$, $b\ge 0$.   
\end{lemma}

\begin{proof}
We set $z=a/b\ge 1$. To verify \eqref{2.ineq1}, we need to show that for all $z\ge 1$,
\begin{align*}
  \bigg(2-\frac{1}{s}\bigg)(z^s-1)
  \ge (z^{s-1}-1)(z+1) \ge \bigg(1-\frac{1}{s}\bigg)(z^s-1).
\end{align*}
The first inequality is equivalent to
\begin{align*}
  F_1(z) := \bigg(1-\frac{1}{s}\bigg)z^s - z^{s-1} + z 
  + \frac{1}{s} - 1 \ge 0.
\end{align*}
Since $F_1(1)=0$ and $F_1'(z) = (s-1)z^{s-2}(z-1) + 1\ge 1$, this yields $F_1(z)\ge 0$ for $z\ge 1$. The second inequality is equivalent to 
\begin{align*}
  F_2(z) := \frac{1}{s}(z^s-1) + z^{s-1} - z \ge 0.
\end{align*}
We claim that $F_2'(z)\ge 0$ for $z\ge 1$. For this, we write $F_2'(z) = G(z)-1$, where $G(z):=z^{s-2}(z+s-1)$. Then $G'(z) = (s-1)z^{s-3}(z + s - 2)$. Since $z\ge 1$ and $s\ge 1$, we have $z+s-2\ge 0$ and hence $G'(z)\ge 0$. Taking into account that $G(1)=s\ge 1$, we infer that $G(z)\ge 1$. Consequently, $F_2(z)=G(z)-1\ge 0$ for $z\ge 1$. 

Inequality \eqref{2.ineq12} for $1<s\le 2$ can be formulated as
\begin{align*}
  z^s-1 \ge (z^{s-1}-1)(z+1)\quad\mbox{for }z = a/b\ge 1.
\end{align*}
This follows directly from $z^{s-1}\le z$ (here, we need that $s\le 2$) and thus
\begin{align*}
  (z^{s-1}-1)(z+1) = z^s + z^{s-1} - z - 1 \le z^s-1.
\end{align*}
If $a<b$, we set $z=b/a>1$ and argue in a similar way as before.

Finally, we write inequality \eqref{2.ineq2} for $s\ge 2$ and $z=a/b\ge 1$ as
\begin{align*}
  z^s-1 \ge \frac{s}{2(s-1)}(z^{s-1}-1)(z+1).
\end{align*}
Define the function
\begin{align*}
  F_3(z) := z^s -1 - \frac{s}{2(s-1)}(z^{s-1}-1)(z+1).
\end{align*}
Since $F_3(1)=0$, it is sufficient to show that $F_3'(z)\ge 0$. For this, we compute
\begin{align*}
  F_3'(z) &= \frac{s}{2(s-1)}
  \big((s-2)z^{s-1} - (s-1)z^{s-2} + 1\big), \\
  F_3''(z) &= \frac{s(s-2)}{2}z^{s-3}(z-1) \ge 0.
\end{align*}
It follows from $F_3'(1)=0$ that $F_3'(z)\ge 0$ and consequently $F_3(z)\ge 0$. The case $z=b/a\ge 1$ is treated in a similar way.
\end{proof}

We show in Appendix \ref{sec.fvm} how the discrete chain rules can be applied to two-point flux approximation finite-volume methods. This is somehow natural, since these methods may be viewed as conservative finite-difference schemes written in flux form, and the underlying nonlocal operator naturally involves the finite difference $P_i(u(y))-P_i(u(x))$. 


\section{Existence of a global strong solution}\label{sec.ex}

In this section, we prove Theorem \ref{thm.ex}. We first regularize the equations, prove the existence of a solution to the regularized problem, derive uniform estimates, and then pass to the de-regularization limit. 

\subsection{Approximate problem}

We start by regularizing the kernel $J$ and the initial data $u_i^0$. Let $\eta\in(0,1)$ be the regularization parameter, and let $J_\eta$ be a family of approximating kernels satisfying Assumption (A2), $J_\eta\in W^{1,1}(\Omega)$, and $J_\eta\to J$ strongly in $L^q(\Omega)$ for all $1\le q<\infty$ as $\eta\to 0$. Let $u_{i,\eta}^0\in W^{1,\infty}(\Omega)$ be such that $u_{i,\eta}^0\to u_i^0$ strongly in $L^p(\Omega)$ for all $1\le p<\infty$. As in \cite[Sec.~3]{GaVe19}, we choose the families such that $(J_\eta)$, $(u_\eta^0)$ are bounded in $L^\infty(\Omega)$ and
\begin{align*}
  \|\na J_\eta\|_{L^1(\Omega)}\to \mathrm{TV}(J), \quad
  \|\na u_{i,\eta}^0\|_{L^1(\Omega)}\to\mathrm{TV}(u_{i}^0),
\end{align*}
where $\mathrm{TV}(g)$ denotes the total variation of the function $g$, defined by
\begin{align*}
  \mathrm{TV}(g) = \sup\bigg\{\int_\Omega g\diver\phi dx:
  \phi\in C_c^1(\Omega;\R^d),\ \|\phi\|_{L^\infty(\Omega)}\le 1\bigg\}.
\end{align*}
Let $m_\eta(x) = \int_\Omega J_\eta(x-y)dy$. In view of the uniform $L^1(\Omega)$ bound of $J_\eta$ and Assumption (A2), we have $J_*\le m_\eta(x)\le J^*$ for $x\in\Omega$, $\eta\in(0,1)$, up to a redefinition of $J_*$ and $J^*$. Given $k\in\N$ and $u^{k-1}_i\in L^\infty(\Omega)$ with $u_i^{k-1}\ge 0$ in $\Omega$, we introduce the implicit Euler scheme with varying time step $\tau_k>0$:
\begin{align}\label{3.approx}
  u_i(x) = u_i^{k-1} + \tau_k\int_\Omega J_\eta(x-y)
  \big(P_i(u^+(y))-P_i(u^+(x))\big)dy
 + \tau_k u_i^+(x)f_i(u^+(x))
\end{align}
for $x\in\Omega$, $i=1,\ldots,n$, where $z^+=\max\{0,z\}$ denotes the positive part of $z\in\R$. 

We prove the existence of a solution to \eqref{3.approx} by using the Banach fixed-point theorem. The proof is similar to \cite[Sec.~3.1]{GaVe19}, but since our nonlinearity is different, we need to detail the contraction argument. Let $N\in\N$, $M_0=\max_{i=1,\ldots,n}\|u_{i,\eta}^0\|_{L^\infty(\Omega)}$, and $M_k=M_0\sum_{\ell=0}^k 2^{-\ell}$ for $k=1,\ldots,N$. For the following arguments, we do not need $M_0$ to be independent of $\eta$. We define the function space
\begin{align*}
  V_k = \big\{v\in W^{1,\infty}(\Omega;\R^n):\|v_i\|_{L^\infty(\Omega)}\le M_k
  \mbox{ for }i=1,\ldots,n\big\}.
\end{align*}

\begin{lemma}\label{lem.approx}
Let $u^{k-1}\in V_{k-1}$ be given. Then, for sufficiently small $\tau_k>0$, there exists a unique solution $u\in V_k$ to \eqref{3.approx}. 
\end{lemma}

\begin{proof}
Let $k\in\{1,\ldots,N\}$ and $i\in\{1,\ldots,n\}$ be fixed. We define the fixed-point operator $Q:V_k\to V_k$ by
\begin{align*}
  (Q(v))(x) = u_i^{k-1}(x) + \tau_k\int_\Omega J_\eta(x-y)
  \big(P_i(v^+(y))-P_i(v^+(x))\big)dy + \tau_k v_i^+(x) f_i(v^+(x)).
\end{align*}
The operator is well-posed provided that $Q(V_k)\subset V_k$ holds. To see this, we observe that $Q(v)\in W^{1,\infty}(\Omega)$, since the functions are sufficiently smooth, and we estimate 
\begin{align*}
  |Q(v(x))| &\le u_i^{k-1}(x) + \tau_k\int_\Omega J_\eta(x-y)
  \big(|P_i(v^+(y))|+|P_i(v^+(x))|\big)dy \\
  &\phantom{xx}+ \tau_k v_i^+(x)
  \bigg(b_{i0} + \sum_{j=1}^n b_{ij}v_j^+(x)\bigg) \\
  &\le M_{k-1} + CJ^*\tau_k\|v_i^+\|_{L^\infty(\Omega)}
  \bigg(1 + \sum_{j=1}^n\|v_j^+\|_{L^\infty(\Omega)}^s\bigg) \\
  &\le M_{k-1} + C\tau_k M_k(1+M_k^s). 
\end{align*}
Choosing $C_0\ge 2C(1+2^{s}M_0^s)$ and $0<\tau_k<1/(C_0 2^{k})$, we infer from $M_k\le 2M_0$ that 
\begin{align*}
  |Q(v(x))| \le M_{k-1} + 2C\tau_k M_0(1+2^s M_0^s) 
  \le M_{k-1} + M_0 2^{-k} = M_k.
\end{align*}
This shows that $Q(V_k)\subset V_k$. As $P_i$ and $u\mapsto u_if_i(u)$ are locally Lipschitz continuous and functions in $V_k$ are bounded, we obtain for $v$, $w\in V_k$ that
\begin{align*}
  |Q(v(x))-Q(w(x))| \le C\tau_k\|v-w\|_{L^\infty(\Omega)},
\end{align*}
where $C>0$ depends on the Lipschitz constants of $P_i$ and $u\mapsto u_if_i(u)$ on $V_k$. Thus, choosing $\tau_k<1/C$, the mapping $Q$ becomes a contraction, and Banach's fixed-point theorem yields a unique fixed point $u\in V_k$ to \eqref{3.approx}.
\end{proof}

We claim that the solution to \eqref{3.approx} is nonnegative. For this, let $z^-=\max\{0,-z\}$ for $z\in\R$. Then $z=z^+ - z^-$. We integrate \eqref{3.approx} over $\{u_i<0\}$, use $P_i(u^+) = u_i^+p_i(u^+) = 0$ in $\{u_i<0\}$, and observe that $u_i^{k-1}\ge 0$ as well as $J_\eta\ge 0$:
\begin{align*}
  \int_{\{u_i<0\}}& u_i^-dx = \int_{\{u_i<0\}} u_i^+ dx
  - \int_{\{u_i<0\}} u_i^{k-1}dx
  - \tau_k\int_{\{u_i<0\}}u_i^+f_i(u^+)dx \\
  &\phantom{xx}- \tau_k\int_\Omega
  \int_\Omega J_\eta(x-y)\big(u_i^+(y)p_i(u^+(y))
  - u_i^+(x)p_i(u^+(x))\big)\mathrm{1}_{\{u_i(x)<0\}}dydx \\
  &= -\int_{\{u_i<0\}} u_i^{k-1}dx
  - \tau_k\int_{\{u_i<0\}}\int_\Omega J_\eta(x-y)
  u_i^+(y)p_i(u^+(y))dydx \le 0.
\end{align*}
This implies that $0\le\int_{\{u_i<0\}} u_i^-(x)dx\le 0$ and hence $u_i\ge 0$ in $\Omega$. 

Notice that it may be necessary to choose progressively smaller values of $\tau_k$ as the time increases. Therefore, the final time $T$ cannot be arbitrarily large. In the following, we derive a priori estimates that allow us to continue the solution to any arbitrarily large final time.

\subsection{Approximate entropy inequalities}

Uniform estimates are derived from the following approximate entropy inequalities. We suppose that Assumptions (A1)--(A4) hold. 

\begin{lemma}[Approximate entropy inequality I]\label{lem.aei1}
Let Assumption (A5a) hold and let $u$ be the strong solution to \eqref{1.eq}--\eqref{1.ic} constructed in Lemma \ref{lem.approx}. Then
\begin{align}\label{3.aei1}
  \int_\Omega &(h(u)-h(u^{k-1}))dx + \frac{2\tau_k}{s}\sum_{i=1}^n a_{i0}
  \int_\Omega\int_\Omega J_\eta(x-y)
  \big(u_i(x)^{s/2}-u_i(y)^{s/2}\big)^2 dydx \\
  &\phantom{xx}+ \frac{s\tau_k}{8(s-1)C(s)}\sum_{i=1}^n
  \int_\Omega\int_\Omega J_\eta(x-y)
  \bigg(a_{ii}-C(s)\sum_{j=1,\neq i}^n a_{ij}\bigg)
  \big(u_i(x)^s - u_i(y)^s\big)^2 dydx \nonumber \\
  &\le C\tau_k\bigg(1 + \int_\Omega h(u)dx\bigg), \nonumber 
\end{align}
recalling definition \eqref{1.Cs} of the constant $C(s)$.
\end{lemma}

\begin{proof}
To simplify the notation, we set $X_i=u_i(x)$ and $Y_i=u_i(y)$ for $i=1,\ldots,n$. We multiply \eqref{3.approx} by $h_i'(X_i) = (sX_i^{s-1}-1)/(s-1)$, sum over $i=1,\ldots,n$, and integrate over $\Omega$. Then $I_0+I_1+I_2=0$, where 
\begin{align}\label{3.I12} 
  I_0 &= \frac{1}{\tau_k}\sum_{i=1}^n\int_\Omega
  (X_i-u_i^{k-1}(x))h'_i(X_i)dx, \\
  I_1 &= -\sum_{i=1}^n\int_\Omega\int_\Omega J_\eta(x-y)
  \big(P_i(u(y))-P_i(u(x))\big)h_i'(X_i)dydx, \nonumber \\
  I_2 &= -\sum_{i=1}^n\int_\Omega
  \bigg(b_{i0}-\sum_{j=1}^n b_{ij}X_j\bigg)X_i h'_i(X_i)dx. \nonumber  
\end{align}
Since $h_i$ is convex, the first term $I_0$ can be estimated according to
\begin{align*}
  I_0 \ge \frac{1}{\tau_k}\int_\Omega\big(h(u)-h(u^{k-1})\big)dx.
\end{align*}
Taking into account that $J_\eta$ is even, we obtain
\begin{align*}
  \sum_{i=1}^n\int_\Omega\int_\Omega J_\eta(x-y)
  \big(P_i(u(y))-P_i(u(x))\big)dydx = 0,
\end{align*}
which shows that
\begin{align*}
  I_1 &= -\frac12\sum_{i=1}^n\int_\Omega\int_\Omega J_\eta(x-y)
  \big(P_i(u(y))-P_i(u(x))\big)\big(h_i'(X_i)-h_i'(Y_i)\big)dydx \\
  &= \frac{s}{2(s-1)}\sum_{i=1}^n\int_\Omega\int_\Omega J_\eta(x-y)
  \big(P_i(u(x))-P_i(u(y))\big)(X_i^{s-1}-Y_i^{s-1})dydx.
\end{align*}
We insert definition \eqref{1.p} of the pressure:
\begin{align*}
  I_1 &= \frac{s}{2(s-1)}\sum_{i=1}^n a_{i0}
  \int_\Omega\int_\Omega J_\eta(x-y)
  (X_i-Y_i)(X_i^{s-1}-Y_i^{s-1})dydx \\
  &\phantom{xx}+ \frac{s}{2(s-1)}\sum_{i,j=1}^n a_{ij}
  \int_\Omega\int_\Omega J_\eta(x-y)
  (X_iX_j^s - Y_iY_j^s)(X_i^{s-1}-Y_i^{s-1})dydx.
\end{align*}
Setting
\begin{align}
  I_{11} &= \sum_{i=1}^n a_{i0}(X_i-Y_i)(X_i^{s-1}-Y_i^{s-1}), 
  \nonumber \\
  I_{12} &= \sum_{i,j=1,\,i\neq j}^n a_{ij}
  (X_iX_j^s - Y_iY_j^s)(X_i^{s-1}-Y_i^{s-1}), \label{3.I123} \\
  I_{13} &= \sum_{i=1}^n a_{ii}
  (X_i^{s+1} - Y_i^{s+1})(X_i^{s-1}-Y_i^{s-1}), \nonumber 
\end{align}
we need to estimate 
\begin{align}\label{3.I1}
  I_1 = \frac{s}{2(s-1)}\int_\Omega\int_\Omega J_\eta(x-y)
  (I_{11}+I_{12}+I_{13})dydx.
\end{align}

We apply Lemma \ref{lem.chain1} with $\beta=s-1>0$ and $\gamma=1$:
\begin{align*}
  I_{11} \ge \frac{4(s-1)}{s^2}\sum_{i=1}^n a_{i0}
  (X_i^{s/2}-Y_i^{s/2})^2.
\end{align*}
Next, we estimate $I_{13}$ by applying Lemma \ref{lem.chain1} with $\beta=s+1$ and $\gamma=s-1$:
\begin{align*}
  I_{13} &\ge \frac{(s+1)(s-1)}{s^2}\sum_{i=1}^n a_{ii}
  (X_i^s-Y_i^s)^2 \\
  &= \frac{(s+1)(s-1)}{s^2}\sum_{i=1}^n a_{ii}
  \bigg(\frac{1}{\delta}(X_i^s-Y_i^s)^2
  + \frac{\delta-1}{\delta}(X_i^s-Y_i^s)^2\bigg),
\end{align*} 
where $\delta>1$. Now, we distinguish between the cases $1<s<2$ and $s\ge 2$. Let $1<s<2$ and $\delta=2$. Inequality \eqref{2.ineq12} for $a=X_i$ and $b=Y_i$ yields
\begin{align*}
  I_{13} \ge \frac{(s+1)(s-1)}{2s^2}\sum_{i=1}^n a_{ii}\big(
  (X_i^{s-1}-Y_i^{s-1})^2(X_i+Y_i)^2 + (X_i^s-Y_i^s)^2\big).
\end{align*}
Let $s\ge 2$ and $\delta=s^2/(4(s-1)^2)+1$. Then, by inequality \eqref{2.ineq2},
\begin{align*}
  I_{13} &\ge \frac{(s+1)(s-1)}{s^2}\sum_{i=1}^n a_{ii}\bigg(
  \frac{s^2}{4\delta(s-1)^2}
  (X_i^{s-1}-Y_i^{s-1})^2(X_i+Y_i)^2
  + \frac{\delta-1}{\delta}(X_i^s-Y_i^s)^2\bigg) \\
  &\ge \min\bigg\{\frac{s+1}{4\delta(s-1)},
  \frac{\delta-1}{\delta}\frac{(s+1)(s-1)}{s^2}\bigg\} \\
  &\phantom{xx}\times
  \sum_{i=1}^n a_{ii}\big((X_i^{s-1}-Y_i^{s-1})^2(X_i+Y_i)^2
  + (X_i^s-Y_i^s)^2\big) \\
  &= \frac{(s+1)(s-1)}{4(s-1)^2+s^2}
  \sum_{i=1}^n a_{ii}\big((X_i^{s-1}-Y_i^{s-1})^2(X_i+Y_i)^2
  + (X_i^s-Y_i^s)^2\big),
\end{align*}
where the last step is a consequence of the choice of $\delta$, leading to $(s+1)/(4\delta(s-1)) = (s+1)(s-1)(\delta-1)/(\delta s^2)$. Summarizing both cases $1<s<2$ and $s\ge 2$, we arrive at
\begin{align*}
  I_{13}\ge \frac{1}{4C(s)}\sum_{i=1}^n a_{ii}\big((X_i^{s-1}-Y_i^{s-1})^2(X_i+Y_i)^2
  + (X_i^s-Y_i^s)^2\big),
\end{align*}
where $C(s)$ is defined in \eqref{1.Cs}. This expression is used to compensate terms coming from $I_{12}$. It follows from
\begin{align*}
  X_iX_j^s - Y_iY_j^s = \frac12(X_i-Y_i)(X_j^s+Y_j^s)
  + \frac12(X_i+Y_i)(X_j^s-Y_j^s).
\end{align*}
that 
\begin{align*}
  I_{12} &= \frac12\sum_{i,j=1,\,i\neq j}^n a_{ij}
  (X_i-Y_i)(X_i^{s-1}-Y_i^{s-1})(X_j^s+Y_j^s) \\
  &\phantom{xx}+ \frac12\sum_{i,j=1,\,i\neq j}^n a_{ij}
  (X_i+Y_i)(X_i^{s-1}-Y_i^{s-1})(X_j^s-Y_j^s).
\end{align*}
The first term on the right-hand side is nonnegative, since $z\mapsto z^{s-1}$ is monotone, and will be neglected. Thus, applying Young's inequality to the remaining term and taking into account the symmetry of $a_{ij}$,
\begin{align*}
  I_{12} &\ge -\frac14\sum_{i,j=1,\,i\neq j}^n a_{ij}
  \big((X_i^{s-1}-Y_i^{s-1})^2(X_i+Y_i)^2 + (X_j^s-Y_j^s)^2\big) \\
  &= -\frac14\sum_{i=1}^n\bigg(\sum_{j=1,\,j\neq i}^n a_{ij}\bigg)
  \big((X_i^{s-1}-Y_i^{s-1})^2(X_i+Y_i)^2 + (X_i^s-Y_i^s)^2\big).
\end{align*}

Putting these estimates together, we infer from \eqref{3.I1} that 
\begin{align*}
  I_1 &\le -\frac{2}{s}\int_\Omega\int_\Omega J_\eta(x-y)
  \sum_{i=1}^n a_{i0}(X_i^{s/2}-Y_i^{s/2})^2 dydx \\
  &\phantom{xx}- \frac{s}{8(s-1)C(s)}\sum_{i=1}^n
  \bigg(a_{ii} - C(s)\sum_{j=1,\,j\neq i}^n a_{ij}\bigg) \\
  &\phantom{xx}\times\int_\Omega\int_\Omega J_\eta(x-y)
  \big((X_i^{s-1}-Y_i^{s-1})^2(X_i+Y_i)^2 + (X_i^s-Y_i^s)^2\big)dydx.
\end{align*}
Assumption \eqref{1.a} implies that the last term is nonpositive.

It remains to estimate the reaction term $I_2$. We split $I_2 = I_{21}+I_{22}$, where
\begin{align*}
  I_{21} = \sum_{i=1}^n \frac{b_{i0}}{s-1}\int_\Omega(sX_i^s-X_i)dx,
  \quad I_{22} = -\sum_{i,j=1}^n\frac{b_{ij}}{s-1}\int_\Omega
  X_j(sX_i^s-X_i)dx.
\end{align*}
Writing $(sX_i^s-X_i)/(s-1) = h_i(X_i) + X_i^s$ and using $h_i(X_i)\ge 0$ for $X_i\ge 1$, we find that
\begin{align*}
  I_{21} &= \sum_{i=1}^n b_{i0}\int_\Omega(h_i(X_i)+X_i^s)dx
  \le C + C\int_\Omega h(u)dx, \\
  I_{22} &= -\sum_{i,j=1}^n b_{ij}\bigg(\int_{\{0\le X_i\le 1\}}
  X_j(h_i(X_i) + X_i)dx + \int_{\{X_i>1\}}X_j(h_i(X_i) + X_i)dx\bigg) \\
  &\le -\sum_{i,j=1}^n b_{ij}\int_{\{0\le X_i\le 1\}}
  X_j\bigg(\frac{X_i^s-1}{s-1} + X_i\bigg)dx \\
  &\le C\sum_{j=1}^n\int_\Omega X_j dx \le C + C\int_\Omega h(u)dx.
\end{align*} 
Collecting the estimates for $I_1$ and $I_2$ finishes the proof.
\end{proof}
 
\begin{lemma}[Approximate entropy inequality II]\label{lem.aei2}
Let Assumption (A5b) hold and let $u$ be the strong solution to \eqref{1.eq}--\eqref{1.ic}, constructed in Lemma \ref{lem.approx}. Then there exists a constant $c>0$, depending on the smallest eigenvalue $\alpha$ of $(a_{ij})$, such that
\begin{align}\label{3.aei2}
  \int_\Omega &(h(u)-h(u^{k-1})dx + \frac{2\tau_k}{s}\sum_{i=1}^n a_{i0}
  \int_\Omega\int_\Omega J_\eta(x-y)
  \big(u_i(x)^{s/2}-u_i(y)^{s/2}\big)^2 dydx \\
  &+ c\tau_k\sum_{i=1}^n\int_\Omega\int_\Omega J_\eta(x-y)
  \big(u_i(x)^s - u_i(y)^s\big)^2 dydx
  \le C\tau_k\bigg(1 + \int_\Omega h(u)dx\bigg), \nonumber 
\end{align}
\end{lemma}

\begin{proof}
We proceed as the proof of Lemma \ref{lem.aei2} but we estimate $I_{12}$ and $I_{13}$, defined in \eqref{3.I123}, differently. Recall that
\begin{align*}
  I_{12} = \frac12\sum_{i,j=1,\,i\neq j}^n a_{ij}
  (X_i+Y_i)(X_i^{s-1}-Y_i^{s-1})(X_j^s-Y_j^s).
\end{align*}
We use inequality \eqref{2.ineq1} to find that if $X_i-Y_i$ and $X_j-Y_j$ have the same sign then 
\begin{align*}
  (X_i+Y_i)(X_i^{s-1}-Y_i^{s-1})(X_j^s-Y_j^s)
  \ge \frac{s-1}{s}(X_i^s-Y_i^s)(X_j^s-Y_j^s),
\end{align*}
while if $X_i-Y_i$ and $X_j-Y_j$ have opposite signs then
\begin{align*}
  (X_i+Y_i)(X_i^{s-1}-Y_i^{s-1})(X_j^s-Y_j^s)
  \ge \frac{2s-1}{s}(X_i^s-Y_i^s)(X_j^s-Y_j^s).
\end{align*}
This shows that, for some constant $C>0$,
\begin{align*}
  I_{12} \ge C\sum_{i,j=1,\,i\neq j}^n a_{ij}(X_i^s-Y_i^s)(X_j^s-Y_j^s).
\end{align*}
 The term $I_{13}$ is estimated with the use of Lemma \ref{lem.chain1}:
\begin{align*}
  I_{13}\ge \frac{4(s+1)(s-1)}{s^2}\sum_{i=1}^n a_{ii}(X_i^s-Y_i^s)^2.
\end{align*}
Collecting the estimates for $I_{12}$ and $I_{13}$ shows
\begin{align*}
  I_{12} +I_{13} \ge C\sum_{i,j=1}^n a_{ij}(X_i^s-Y_i^s)(X_j^s-Y_j^s)
  \ge C\alpha \sum_{i=1}^n(X_i^s-Y_i^s)^2,
\end{align*}
recalling that $\alpha>0$ is the smallest eigenvalue of $(a_{ij})$. This finishes the proof.
\end{proof}

\subsection{Further uniform estimates}

The approximate entropy inequalities imply some uniform bounds for the solution $u^k:=u$ to \eqref{3.approx}. Let $t_0=0$ and set $t_k=\sum_{\ell=1}^k\tau_\ell$ for $k=1,\ldots,N$ as well as $\tau=\min_{k=1,\ldots,N}\tau_k>0$. Recall that $t_N=T$. We introduce the piecewise constant and piecewise linear in time functions, respectively, by 
\begin{align*}
  u_i^{(\tau)}(t,x) = u_i^k, \quad
  \widetilde{u}_i^{(\tau)}(t,x) = u_i^k(x) + \frac{t_k-t}{\tau_k}
  (u_i^{k-1}(x)-u_i^k(x))
\end{align*}
for $x\in\Omega$ and $t\in(t_{k-1},t_k]$. Then \eqref{3.approx} can be written as
\begin{align}\label{3.tau}
  \pa_t\widetilde{u}_i^{(\tau)}(t,x) 
  = \int_\Omega J_\eta(x-y)\big(P_i(u^{(\tau)}(t,y))
  - P_i(u^{(\tau)}(t,x))\big)dy
  + u_i^{(\tau)}(t,x)f_i(u^{(\tau)}(t,x))
\end{align}
for $x\in\Omega$, $0<t<T$. 

\begin{lemma}[Uniform $L^1(\Omega)$ estimates]\label{lem.est0}
For sufficiently small $\tau>0$, there exists a constant $C>0$ independent of $\tau$ and $\eta$ such that for all $i=1,\ldots,n$,
\begin{align*}
  \|u_i^{(\tau)}\|_{L^\infty(0,T;L^1(\Omega))}\le C.
\end{align*}
\end{lemma}

\begin{proof}
We integrate \eqref{3.approx} over $\Omega$ and use the fact that $J_\eta$ is an even function:
\begin{align*}
  \int_\Omega u_i^k(x)dx &= \int_\Omega u_i^{k-1}dx
  + \tau_k\int_\Omega\int_\Omega J_\eta(x-y)
  \big(P_i(u^{(\tau)}(t,y)) - P_i(u^{(\tau)}(t,x))\big)dydx \\
  &\phantom{xx}+ \tau_k\int_\Omega u_i^k(x)f_i(u^k(x))dx \\
  &= \int_\Omega u_i^{k-1}dx
  + \tau_k\int_\Omega u_i^k(x)
  \bigg(b_{i0} - \sum_{j=1}^n b_{ij}u_j^k(x)\bigg)dx \\
  &\le \int_\Omega u_i^{k-1}dx + \tau_k b_{i0}\int_\Omega u_{i}^kdx.
\end{align*}
A summation over $i=1,\ldots,n$ and $k=1,\ldots,\ell$ for $\ell\le N$ gives
\begin{align*}
  \sum_{i=1}^n\int_\Omega u_i^\ell(x)dx
  \le \sum_{i=1}^n\int_\Omega u_i^0(x)dx
  + \sum_{k=1}^\ell\tau_k\sum_{i=1}^n b_{i0}\int_\Omega u_i^k dx
  \le C(T,u^0).
\end{align*}
Choosing $\tau>0$ sufficiently small, the last term can be absorbed by the left-hand side, and the discrete Gronwall inequality (see, e.g. \cite{Cla87}) finishes the proof.
\end{proof}

Lemma \ref{lem.est0} is valid without the use of the approximate entropy inequalities and consequently without Assumptions (A5a) and (A5b). We use the entropy inequalities to achieve further uniform estimates.

\begin{lemma}[Uniform $L^s(\Omega)$ estimates]\label{lem.est1}
For sufficiently small $\tau>0$, there exists a constant $C>0$ independent of $\tau>0$ and $\eta$ such that for all $i=1,\ldots,n$,
\begin{align*}
  \|u_i^{(\tau)}\|_{L^\infty(0,T;L^s(\Omega))}
  + \|u_i^{(\tau)}\|_{L^{2s}(\Omega_T)}
  + a_{i0}^{1/s}\|u_i^{(\tau)}\|_{L^{s}(\Omega_T)} \le C.
\end{align*}
\end{lemma}

\begin{proof}
We expand the square in the last term on the left-hand side of \eqref{3.aei1} or \eqref{3.aei2} and use again the evenness of $J_\eta$:
\begin{align*}
  \int_\Omega&\int_\Omega J_\eta(x-y)\big(u_i^k(x)^s-u_i^k(y)^s\big)^2
  dxdy \\
  &= 2\int_\Omega\bigg(\int_\Omega J_\eta(x-y)dy\bigg)
  u_i^k(x)^{2s} dx
  - 2\int_\Omega\int_\Omega J_\eta(x-y)u_i^k(x)^s u_i^k(y)^s dydx \\
  &\ge 2J_*\|(u_i^k)^s\|_{L^2(\Omega)}^2 
  - 2\|J_\eta\|_{L^\infty(\Omega)}\|(u_i^k)^s\|_{L^1(\Omega)}^2
  \le 2J_*\|u_i^k\|_{L^{2s}(\Omega)}^{2s}
  - C\|u_i^k\|_{L^s(\Omega)}^{2s},
\end{align*}
where $C\ge 2\|J_\eta\|_{L^\infty(\Omega)}$ can be chosen to be independent of $\eta$. We claim that the right-hand side can be bounded independently of $\eta$. Indeed, by the interpolation inequality, with $1/s = \theta/(2s) + (1-\theta)$,
\begin{align*}
  \|u_i^k\|_{L^s(\Omega)}^{2s}
  \le \|u_i^k\|_{L^{2s}(\Omega)}^{2s\theta}
  \|u_i^k\|_{L^1(\Omega)}^{2s(1-\theta)}
  \le J_*\|u_i^k\|_{L^{2s}(\Omega)}^{2s} 
  + C(J_*)\|u_i^k\|_{L^1(\Omega)}^{2s}.
\end{align*}
We deduce from Lemma \ref{lem.est0}, for some constant $C>0$ independent of $\tau$ and $\eta$, that
\begin{align*}
  \int_\Omega\int_\Omega J_\eta(x-y)\big(u_i^k(x)^s-u_i^k(y)^s\big)^2
  dydx \ge J_*\|u_i^k\|_{L^{2s}(\Omega)}^{2s} - C.
\end{align*}
Similar arguments lead to 
\begin{align*}
  \int_\Omega\int_\Omega J_\eta(x-y)
  \big(u_i^k(x)^{s/2}-u_i^k(y)^{s/2}\big)^2 dxdy
  \ge J_*\|u_i^k\|_{L^{s}(\Omega)}^{s} - C.
\end{align*}
Thus, the approximate entropy inequalities yield
\begin{align*}
  \int_\Omega&(h(u^k)-h(u^{k-1}))dx
  + c\tau_k\sum_{i=1}^n\big(\|u_i^k\|_{L^{2s}(\Omega)}^{2s}
  + \|u_i^k\|_{L^{s}(\Omega)}^{s}\big)
  \le C + C\tau_k\int_\Omega h(u^k)dx,
\end{align*}
where $c>0$, $C>0$ do not depend on $\tau$ or $\eta$. We sum over $k=1,\ldots,N$ and apply the discrete Gronwall inequality (for sufficiently small $\tau>0$) to find that 
\begin{align*}
  \int_\Omega h(u^{(\tau)}(t)) dx
  + c\int_0^T\big(\|u_i^{(\tau)}\|_{L^{2s}(\Omega)}^{2s}
  + \|u_i^{(\tau)}\|_{L^{s}(\Omega)}^{s}\big)dt
  \le C(T) + C\tau\int_\Omega h(u^0)dx,
\end{align*}
finishing the proof.
\end{proof}

The previous lemma allows us to derive further uniform bounds.

\begin{lemma}[Uniform $L^\infty(\Omega)$ estimates]\label{lem.est2}
For sufficiently small $\tau>0$, there exists a constant $C>0$ independent of $\tau>0$ and $\eta$ such that for all $i=1,\ldots,n$,
\begin{align*}
  \|u_i^{(\tau)}\|_{L^\infty(\Omega_T)} 
  + \|\widetilde{u}_i^{(\tau)}\|_{L^\infty(\Omega_T)}
  + \|\na u_i^{(\tau)}\|_{L^\infty(0,T;L^1(\Omega))}
  + \|\pa_t\widetilde{u_i}^{(\tau)}\|_{L^\infty(\Omega_T)} \le C.
\end{align*}
Furthermore, there exists a constant $C(\eta)>0$ independent of $\tau$ such that
\begin{align*}
  \|\na u_i^{(\tau)}\|_{L^\infty(\Omega_T)}
  + \|\na\widetilde{u}_i^{(\tau)}\|_{L^\infty(\Omega_T)}\le C(\eta).
\end{align*}
\end{lemma}

\begin{proof}
We use the nonnegativity of $u_i^k$ and $u_i^{k-1}$ and the Cauchy--Schwarz inequality:
\begin{align*}
  u_i^k(x)&= u_i^{k-1}(x) + \tau_k\int_\Omega J_\eta(x-y)
  \big(P_i(u^k(y))-P_i(u^k(x))\big)dy 
  + \tau_ku_i^k(x)f_i(u^k(x)) \\
  &\le u_i^{k-1}(x) + \tau_k\|J_\eta\|_{L^\infty(\Omega)}
  \int_\Omega|P_i(u^k(y))|dy 
  + b_{i0}\tau_k u_i^k(x) \\
  &\le u_i^{k-1}(x) + C\tau_k
  \bigg(\|u_i^k\|_{L^1(\Omega)} + \|u_i^k\|_{L^2(\Omega)}
  \sum_{j=1}^n\|(u_j^k)^s\|_{L^2(\Omega)}\bigg) 
  + b_{i0}\tau_k u_i^k(x).
\end{align*}
It follows that
\begin{align*}
  (1-b_{i0}\tau_k)\|u_i^k\|_{L^\infty(\Omega)}
  \le \|u_i^{k-1}\|_{L^\infty(\Omega)}
  + C\tau_k\bigg(\|u_i^k\|_{L^1(\Omega)} 
  + \|u_i^k\|_{L^2(\Omega)}
  \sum_{j=1}^n\|u_j^k\|_{L^{2s}(\Omega)}^s\bigg).
\end{align*}
By the induction hypothesis, $u_i^{k-1}\in L^\infty(\Omega)$. Thus, choosing $\tau_k<1/b_{i0}$ for all $i=1,\ldots,n$ and taking into account Lemma \ref{lem.est1}, we conclude the $L^\infty(\Omega)$ bound for $u_i^k$. This estimate is independent of $\tau$ and $\eta$. By iteration, this gives an $L^\infty(\Omega_T)$ bound for $u_i^{(\tau)}$ and $\widetilde{u}_i^{(\tau)}$. 

For the gradient estimate, we differentiate \eqref{3.approx}:
\begin{align}\label{3.aux} 
  |\na u_i^k(x)| &\le |\na u_i^{k-1}(x)|
  + \tau_k\int_\Omega|\na J_\eta(x-y)||P_i(u^k(y))-P_i(u^k(x))|dy \\
  &\phantom{xx}+ \tau_k|\na P_i(u^k(x))|\int_\Omega J_\eta(x-y)dy
  + \tau_k|\na u_i^k(x)f_i(u^k(x))| \nonumber \\
  &\phantom{xx}+ \tau_k|u_i^k(x)\na f_i(u^k(x))| \nonumber \\
  &\le \|\na u_i^{k-1}\|_{L^\infty(\Omega)}
  + C\tau_k\|\na J_\eta\|_{L^1(\Omega)}\|P_i(u^k)\|_{L^\infty(\Omega)} 
  \nonumber \\
  &\phantom{xx}+ C\big(\|u^k\|_{L^\infty(\Omega)},J^*\big)\tau_k
  \sum_{j=1}^n|\na u_j^k(x)|. \nonumber 
\end{align}
We take the supremum over $x\in\Omega$, sum over $i=1,\ldots,n$, and choose $\tau_k>0$ sufficiently small to absorb the last term by the left-hand side. This gives an $L^\infty(\Omega)$ bound for $\na u_i^k$, which is independent of $\tau$ but depends on $\eta$ through $\|\na u_{i,\eta}^0\|_{L^\infty(\Omega)}\le C(\eta)$. When we integrate \eqref{3.aux} over $\Omega$ instead of taking the supremum, we obtain
\begin{align*}
  \|\na u_i^k\|_{L^1(\Omega)} \le \|\na u_i^{k-1}\|_{L^1(\Omega)}
  + C\tau_k\bigg(1 + \sum_{j=1}^n\|\na u_j^k\|_{L^1(\Omega)}\bigg),
\end{align*}
and after summation over $i=1,\ldots,n$ and choosing $\tau_k$ sufficiently small, this leads to an $L^1(\Omega)$ bound for $u_i^k$ uniformly in $\tau$ and $\eta$, since $\|\na u_{i,\eta}^0\|_{L^1(\Omega)}$ is bounded uniformly in $\eta$. 

Finally, we estimate the time derivative:
\begin{align*}
  |\pa_t\widetilde{u}_i^{(\tau)}(t,x)|
  &\le \int_\Omega|J_\eta(x-y)|\big|P_i(u^{(\tau)}(t,y))
  - P_i(u^{(\tau)}(t,x))\big|dy \\
  &\phantom{xx}+ \big|u_i^{(\tau)}(t,x)f_i(u^{(\tau)}(t,x))\big| 
  \le C\big(\|u^{(\tau)}\|_{L^\infty(\Omega_T)},
  \|J_\eta\|_{L^\infty(\Omega)}\big) \le C,
\end{align*}
recalling that $\|J_\eta\|_{L^\infty(\Omega)}$ was chosen independently of $\eta$. This concludes the proof.
\end{proof}

\subsection{Limit $\tau\to 0$}

The uniform bounds of Lemma \ref{lem.est2} imply that $(\widetilde{u}_i^{(\tau)})$ is bounded in $W^{1,\infty}(\Omega_T)$, which allows us to extract subsequences (not relabeled) such that, as $\tau\to 0$,
\begin{align*}
  u_i^{(\tau)}\rightharpoonup^* u_i 
  &\quad\mbox{weakly* in }L^\infty(0,T;W^{1,\infty}(\Omega)), \\
  \widetilde{u}_i^{(\tau)}\rightharpoonup^* \widetilde{u}_i
  &\quad\mbox{weakly* in }W^{1,\infty}(\Omega_T).
\end{align*}
Moreover, the $W^{1,\infty}(\Omega_T)$ bound for $(\widetilde{u}_i^{(\tau)})$ from Lemma \ref{lem.est2} and the compact embedding $W^{1,\infty}(\Omega_T)\hookrightarrow C^0(\overline{\Omega}_T)$ imply that, for a subsequence,
\begin{align*}
  \widetilde{u}_i^{(\tau)}\to \widetilde{u}_i
  \quad\mbox{strongly in }C^0(\overline{\Omega}_T). 
\end{align*}
For $t\in(t_{k-1},t_k]$, we deduce from
\begin{align*}
  \big|u_i^{(\tau)}(t,x)-\widetilde{u}_i^{(\tau)}(t,x)\big|
  = \bigg|\frac{t_k-t}{\tau_k}\big(u_i^{k-1}(x)-u_i^k(x)\big)\bigg|
  \le \tau\|\pa_t\widetilde{u}_i^{(\tau)}\|_{L^\infty(\Omega_T)}
\end{align*}
in the limit $\tau\to 0$ that $u_i=\widetilde{u}_i$ in $\Omega_T$. In particular, $u_i^{(\tau)}\to u_i$ strongly in $L^\infty(\Omega_T)$ and a.e.\ in $\Omega_T$. The uniform bound for the time derivative yields, up to a subsequence, $\pa_t\widetilde{u}_i^{(\tau)}\rightharpoonup \pa_t\widetilde{u}_i = \pa_t u_i$ weakly* in $L^\infty(\Omega_T)$. The established convergences allow us to perform the limit $\tau\to 0$ in the weak formulation of \eqref{3.tau} to infer that $u_i$ solves
\begin{align}\label{3.eta}
  \int_0^T\int_\Omega\pa_t u_i\phi_i dxdt
  = \int_0^T\int_\Omega\bigg(\int_\Omega J_\eta(x-y)
  \big(P_i(u(y))-P_i(u(x))\big)dy + u_if_i(u)\bigg)\phi_i dxdt
\end{align}
for all $\phi_i\in L^1(\Omega_T)$. We can also take the limit $\tau\to 0$ in the approximate entropy inequality \eqref{3.aei1} to obtain
\begin{align}
  \int_\Omega & h(u(t))dx + \frac{2}{s}\sum_{i=1}^n a_{i0}\int_\Omega\int_\Omega J_\eta(x-y)
  \big(u_i(x)^{s/2}-u_i(y)^{s/2}\big)^2 dydx \nonumber \\
   &\phantom{xx}+ \frac{s}{8(s-1)C(s)}\sum_{i=1}^n 
  \int_\Omega\int_\Omega J_\eta(x-y)
  \bigg(a_{ii} - C(s)\sum_{j=1, j\neq i}^n a_{ij}\bigg)
  \big(u_i(x)^s-u_i(y)^s\big)^2 dydx \label{3.etaei} \\
  &\le C\bigg(1+\int_\Omega h(u^0)dx\bigg). \nonumber 
\end{align}
In view of the strong convergence of $u_i^{(\tau)}$ in $L^\infty(\Omega_T)$ and the properties of functions in $BV(\Omega)$ \cite{AFP00}, we have
\begin{align*}
  \|u_i\|_{L^\infty(0,T;BV(\Omega))} \le \liminf_{\tau\to 0}
  \|u_i^{(\tau)}\|_{L^\infty(0,T;W^{1,1}(\Omega))}\le C,
\end{align*}
where $C>0$ is independent of $\eta$. 

Recall that the final time $T>0$ cannot be arbitrarily large. However, in view of the a priori estimates derived so far, by a standard continuation argument, we can extend our solution to the interval $[0,T]$ for an arbitrary $T>0$.

\subsection{Limit $\eta\to 0$}

Let $u^{(\eta)}$ be the solution to \eqref{3.eta} constructed in the previous subsection. Recall that by construction, $(J_\eta)$ is bounded in $L^\infty(\Omega)$. Lemma \ref{lem.est2} implies that $(u_i^{(\eta)})$ is bounded in $L^\infty(\Omega_T)$. We differentiate \eqref{3.eta} with respect to space:
\begin{align*}
  \pa_t\na u_i^{(\eta)}(x)
  &= \int_\Omega\na J_\eta(x-y)\big(P_i(u^{(\eta)}(y))
  - P_i(u^{(\eta)}(x))\big)dy \\
  &\phantom{xx}+ m_\eta(x)\bigg(a_{i0}\na u_i^{(\eta)}(x)
  + \sum_{j=1}^n a_{ij}
  \na\big(u_i^{(\eta)}(x)u_j^{(\eta)}(x)^s\big)\bigg) \nonumber \\
  &\phantom{xx}+ b_{i0}\na u_i^{(\eta)}(x)
  - \sum_{j=1}^n b_{ij}\na\big(u_i^{(\eta)}(x)u_j^{(\eta)}(x)\big), \nonumber 
\end{align*}
recalling definition \eqref{1.meta} of $m_\eta$. As $(\na J_\eta)$ is bounded in $L^1(\Omega)$ by construction, we obtain after integration over $(0,t)$ with $t\le T$:
\begin{align*}
  |\na u_i^{(\eta)}(t,x)| &\le C\|\na J_\eta\|_{L^1(\Omega)}
  \|P_i(u^{(\eta)})\|_{L^\infty(\Omega_T)} \\
  &\phantom{xx}+ C(\|u^{(\eta)}\|_{L^\infty(\Omega_T)})\sum_{j=1}^n
  \int_0^t|\na u_j^{(\eta)}(r,x)|dr + |\na u_{i,\eta}^0(x)| \le C.
\end{align*}
Since $\na u_{i,\eta}^0$ is uniformly bounded in $L^1(\Omega)$, we conclude after integration over $\Omega$ and summation over $i=1,\ldots,n$ from Gronwall's inequality a uniform bound for $\na u_i^{(\eta)}$ in $L^\infty(0,T;L^1(\Omega))$. Thus, taking into account the uniform bound for the time derivative $\pa_t\widetilde{u}_i^{(\eta)}$ from Lemma \ref{lem.est2}, the Aubin--Lions lemma and the uniform $L^\infty(\Omega_T)$ bound imply the existence of a subsequence (not relabeled) such that, as $\eta\to 0$,
\begin{align*}
  \widetilde{u}_i^{(\eta)}\to \widetilde{u}_i 
  &\quad\mbox{strongly in }L^q(\Omega_T)\mbox{ for any }p<\infty, \\
  \pa_t\widetilde{u}_i^{(\eta)}\rightharpoonup \pa_t \widetilde{u}_i
  &\quad\mbox{weakly* in }L^\infty(\Omega_T), \\
  u_i^{(\eta)}\rightharpoonup^* u_i
  &\quad\mbox{weakly* in }L^\infty(0,T;L^1(\Omega)).
\end{align*}
Similarly as in the previous subsection, it follows that $\widetilde{u}_i=u_i$ and that the limit satisfies $u_i\in L^\infty(0,T;BV(\Omega))$. It follows that $u_i\in W^{1,\infty}(0,T;L^\infty(\Omega))\cap C^0([0,T];L^\infty(\Omega)\cap BV(\Omega))$ and that $u_i$ solves \eqref{1.eq}--\eqref{1.ic}. Moreover, we can pass to the limit $\eta\to 0$ in the entropy inequality \eqref{3.etaei}. 


\section{Uniqueness of solutions}\label{sec.unique}

We prove Theorem \ref{thm.unique} by using the duality method. Let $u$ and $\bar u$ be two strong solutions to \eqref{1.eq}--\eqref{1.ic}. Then $U:=u-\bar u$ solves
\begin{align*}
  \pa_t U_i(t,x)
  &= \int_\Omega J(x-y)\big[\big(P_i(u(y))-P_i(\bar u(y))\big)
  - \big(P_i(u(x))-P_i(\bar u(x))\big)\big]dy \\
  &\phantom{xx}+ u_i(t,x)f_i(u(t,x))-\bar u_i(t,x)f_i(\bar u(t,x))
\end{align*}
for $x\in\Omega$, $t>0$, $i=1,\ldots,n$ and the initial condition $U_i(0,x)=0$. We multiply this differential equation by $\phi_i\in W^{1,1}(0,T;L^1(\Omega))$ with $\phi_i(T,x)=0$, integrate over $\Omega$, sum over $i=1,\ldots,n$, integrate by parts in time, exchange $x$ and $y$ in the nonlocal term, and use the evenness of the kernel:
\begin{align*}
  \sum_{i=1}^n&\int_0^T\int_\Omega U_i(t,x)\pa_t\phi_i(t,x)dxdt \\
  &= -\sum_{i=1}^n\int_0^T\int_\Omega\int_\Omega J(x-y)
  \big(P_i(u(t,x))-P_i(\bar u(t,x))\big)
  \big(\phi_i(t,y)-\phi_i(t,x)\big)dydxdt \\
  &\phantom{xx}+ \sum_{i=1}^n\int_0^t\int_\Omega
  \big(u_i(t,x)f_i(u(t,x))-\bar u_i(t,x)f_i(\bar u(t,x))\big)
  \phi_i(t,x)dxdt.
\end{align*}
We deduce from definition \eqref{1.p} of $P_i$ and the symmetry of $a_{ij}$ that
\begin{align*}
  P_i&(u)-P_i(\bar u) = a_{i0}(u_i-\bar u_i)
  + \sum_{j=1}^n a_{ij}\big(u_i u_j^s - \bar u_i\bar u_j^s\big) \\
  &= a_{i0}(u_i-\bar u_i)
  + \frac12\sum_{j=1}^n a_{ij}(u_i-\bar u_i)(u_j^s+\bar u_j^s)
  + \frac12\sum_{j=1}^n a_{ij}(u_i+\bar u_i)(u_j^s-\bar u_j^s) \\
  &= a_{i0}(u_i-\bar u_i)
  + \frac12\sum_{j=1}^n a_{ij}(u_i-\bar u_i)(u_j^s+\bar u_j^s)
  + \frac12\sum_{j=1}^n a_{ij}(u_j+\bar u_j)(u_i^s-\bar u_i^s) \\
  &= a_{i0}U_i
  + \sum_{j=1}^n \frac{a_{ij}}{2}U_i(u_j^s+\bar u_j^s)
  + U_iG_{ij}, \quad\mbox{where} \\
  & G_{ij}(x) = \frac{a_{ij}}{2}(u_j(x)+\bar u_j(x))
  \int_0^1 s\big(\theta u_i(x)+(1-\theta)\bar u_i(x)\big)^{s-1}d\theta.
\end{align*}
Then, using definition \eqref{1.f} of $f_i$,
\begin{align}\label{3.phi}
  & \sum_{i=1}^n\int_0^T\int_\Omega U_i(t,x)\pa_t\phi_i(t,x)dx
  = \sum_{i=1}^n\int_0^T\int_\Omega U_i(t,x)B_i[\phi](t,x)dxdt, 
  \quad\mbox{where} \\
  & B_i[\phi](x) = \int_\Omega J(x-y)(\phi_i(y)-\phi_i(x))
  \bigg(a_{i0} + \sum_{j=1}^n \frac{a_{ij}}{2}
  (u_j^s(x)+\bar u_j^s(x))\bigg)dy \nonumber \\
  &\phantom{B_i[\phi](x)=}+ \int_\Omega J(x-y)\sum_{j=1}^n(\phi_j(y)-\phi_j(x))G_{ij}(x)dy \nonumber \\
  &\phantom{B_i[\phi](x)=}+ 
  \bigg(b_{i0}-\sum_{j=1}^n b_{ij}u_j(x)\bigg)\phi_i(x)
  - \sum_{j=1}^n b_{ji}\bar u_i(x)\phi_j(x). \nonumber 
\end{align}
Let $\phi_i\in L^\infty(0,T;L^1(\Omega))$ be the solution to the problem
\begin{align}\label{3.phi2}
  \pa_t\phi_i = U_i + B_i[\phi], \quad 0<t<T, \quad \phi_i(T,x)=0.
\end{align}
Then problem \eqref{3.phi} reduces to 
\begin{align*}
  \sum_{i=1}^n\int_0^T\int_\Omega U_i(t,x)^2 dxdt = 0,
\end{align*}
which implies that $U_i=u_i-\bar u_i=0$ in $\Omega_T$ and thus proving the uniqueness of solutions.

It remains to verify that there exists a solution $\phi=(\phi_1,\ldots,\phi_n)\in L^\infty(0,T;L^1(\Omega;\R^n))$ to \eqref{3.phi2}. This is done by using the Banach fixed-point theorem. Let $T_0\in[0,T]$ and introduce $X_{T_0} = L^\infty(0,T_0;L^\infty(\Omega;\R^n))$. We define the operator $Q=(Q_1,\ldots,Q_n):X_{T_0}\to X_{T_0}$ by 
\begin{align*}
  Q_i[\phi](t,x) = \int_0^t U_i(r,x)dr + \int_0^t B_i[\phi](r,x)dr, 
  \quad x\in\Omega,\ t>0,\ i=1,\ldots,n.
\end{align*}
We deduce from $U_i$, $J$, $G_{ij}\in L^\infty(\Omega_T)$ that
\begin{align*}
  \|Q[\phi]\|_{X_{T_0}} \le C(T_0 + \|\phi\|_{X_{T_0}})
\end{align*}
and consequently, $Q(X_{T_0})\subset X_{T_0}$, proving the well-posedness of $Q$. We show the contraction property of $Q_i$. Let $\psi$, $\bar\psi\in X_{T_0}$. Then
\begin{align*}
  |Q_i[\psi](x)-Q_i[\bar\psi](x)|
  \le \int_0^t|B_i[\psi-\bar\psi](x)|dr
  \le CT_0\|\psi-\bar\psi\|_{X_{T_0}}
\end{align*}
and, after taking the supremum over $\Omega$,
\begin{align*}
  \|Q[\psi]-Q[\bar\psi]\|_{X_{T_0}} 
  \le CT_0\|\psi-\bar\psi\|_{X_{T_0}}.
\end{align*}
Hence, $Q$ is a contraction if we choose $T_0<1/C$. Banach's fixed-point theorem yields a unique fixed point $\phi\in X_{T_0}$ to \eqref{3.phi2}. As the constant $C>0$ is independent of $T_0$, we can repeat the arguments to obtain eventually a unique solution $\phi$ to \eqref{3.phi2} defined on $[0,T]$. This finishes the proof of Theorem \ref{thm.unique}.

\begin{remark}[Long-time asymptotics]\label{rem.time}\rm 
We show that the solution to \eqref{1.eq}--\eqref{1.ic} with $f_i(u)=0$ and $\kappa_0>0$ (see Theorem \ref{thm.ex}) converges with exponential rate to the constant steady state if $s=2$. We define the constant steady state $u_i^\infty$ by
\begin{align*}
  u_i^\infty = \fint_\Omega u_i^0 dx 
  := \frac{1}{\operatorname{meas}(\Omega)}
  \int_\Omega u_i^0 dx, \quad i=1,\ldots,n. 
\end{align*}
The relative entropy becomes in case $s=2$:
\begin{align*}
  H(u|u^\infty) = \int_\Omega\big(h(u)-h(u^\infty)-h'(u^\infty)
  \cdot(u-u^\infty)\big)dx = \int_\Omega(u-u^\infty)^2 dx,
\end{align*}
where we used the fact that $u$ and $u^\infty$ have the same mass. Then, by inequality \eqref{1.ei} and mass conservation,
\begin{align*}
  \frac{d}{dt}H(u|u^\infty) = \frac{d}{dt}\int_\Omega h(u)dx
  \le -\kappa_0\sum_{i=1}^n\int_\Omega\int_\Omega J(x-y)
  (u_i(x)-u_i(y))^2 dydx.
\end{align*}
We deduce from the nonlocal Poincar\'e inequality \cite[Sec.~4]{CCR06} that
\begin{align*}
  \bigg\|u_i - \fint_\Omega u_i dx\bigg\|_{L^2(\Omega)}^2 
  \le C_P\int_\Omega\int_\Omega J(x-y)(u_i(x)-u_i(y))^2 dx,
\end{align*}
leading to, because of $\fint u_idx = u_i^\infty$,
\begin{align*}
  \frac{d}{dt}H(u|u^\infty) \le -\frac{\kappa_0}{C_P}
  \sum_{i=1}^n\int_\Omega(u_i-u^\infty_i)^2 dx
  = -\frac{\kappa_0}{C_P}H(u|u^\infty).
\end{align*}
Then Gronwall's inequality implies that $H(u(t)|u^\infty)\le H(u^0|u^\infty)e^{-\kappa_0 t/C_P}$ for $t>0$.
\end{remark}

\section{Numerical experiment}\label{sec.numer}

We illustrate the convergence to the steady state, the localization limit, and the behavior of the solutions for different values of the exponent $s$ by numerical experiments for the two-species model in one space dimension. We have chosen $\Omega=(0,1)$, neglected the reaction terms ($b_{i0}=b_{ij}=0$), and taken the kernel function
\begin{align*}
  J_\eps(x) = \frac{K}{\eps}J\bigg(\frac{x}{\eps}\bigg),
  \quad\mbox{where }K^{-1} = \frac12\int_0^1 J(x)x^2dx,\quad
  J = \mathrm{1}_{[-0.5,0.5]},
\end{align*}
where $\eps>0$ is the localization parameter. Equations \eqref{1.eq} are discretized by a semi-implicit scheme, where the integral is approximated by a rectangular rule. The spatial grid size is $\Delta x = 0.02$, and the time step size equals $\Delta t = 0.004$. The diffusion coefficients are
\begin{align*}
  a_{10}=a_{20}=2, \quad a_{11} = 1,\quad a_{12}=a_{21}=0.4, 
  \quad a_{22} = 1.5.
\end{align*}
In particular, the matrix $(a_{ij})$ is symmetric positive definite. We have taken the initial data
\begin{align*}
  u_1^0(x) = \frac{1}{K_1}\exp\bigg(-\frac{(x-0.45)^2}{0.02}\bigg),
  \quad 
  u_2^0(x) = \frac{1}{K_2}\exp\bigg(-\frac{(x-0.65)^2}{0.02}\bigg),
\end{align*}
and $K_1>0$, $K_2>0$ are chosen such that the initial data is normalized. 

Figure \ref{fig.loc} shows the solutions for various values of $\eps$ and for $s=1.2$. We observe that for small $\eps$, the nonlocal solution approximates the local solution accurately. Moreover, the solutions approach, as expected, the constant steady state as $t\to\infty$. 

\begin{figure}[ht]
\includegraphics[width=65mm]{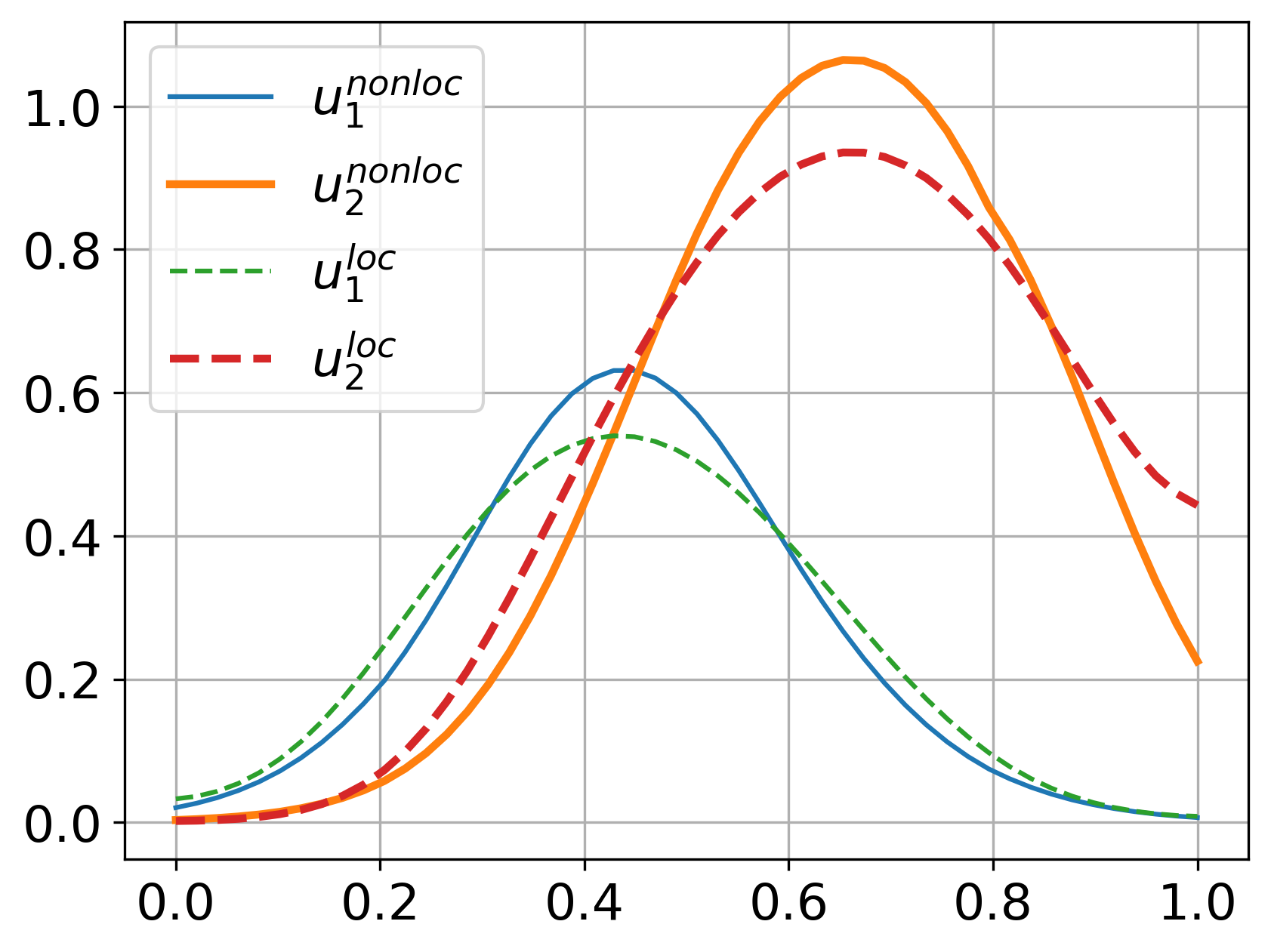}
\includegraphics[width=65mm]{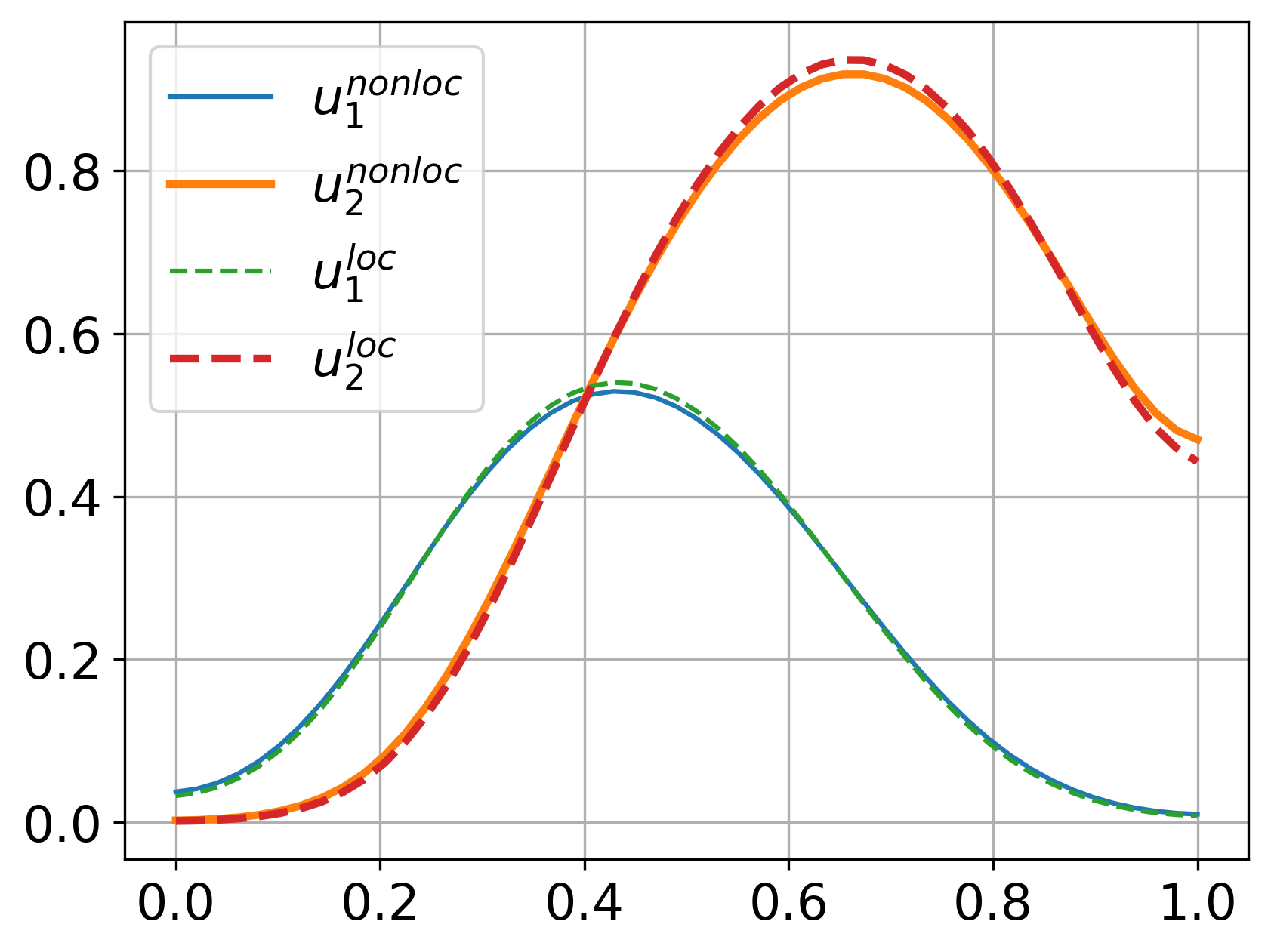}
\includegraphics[width=65mm]{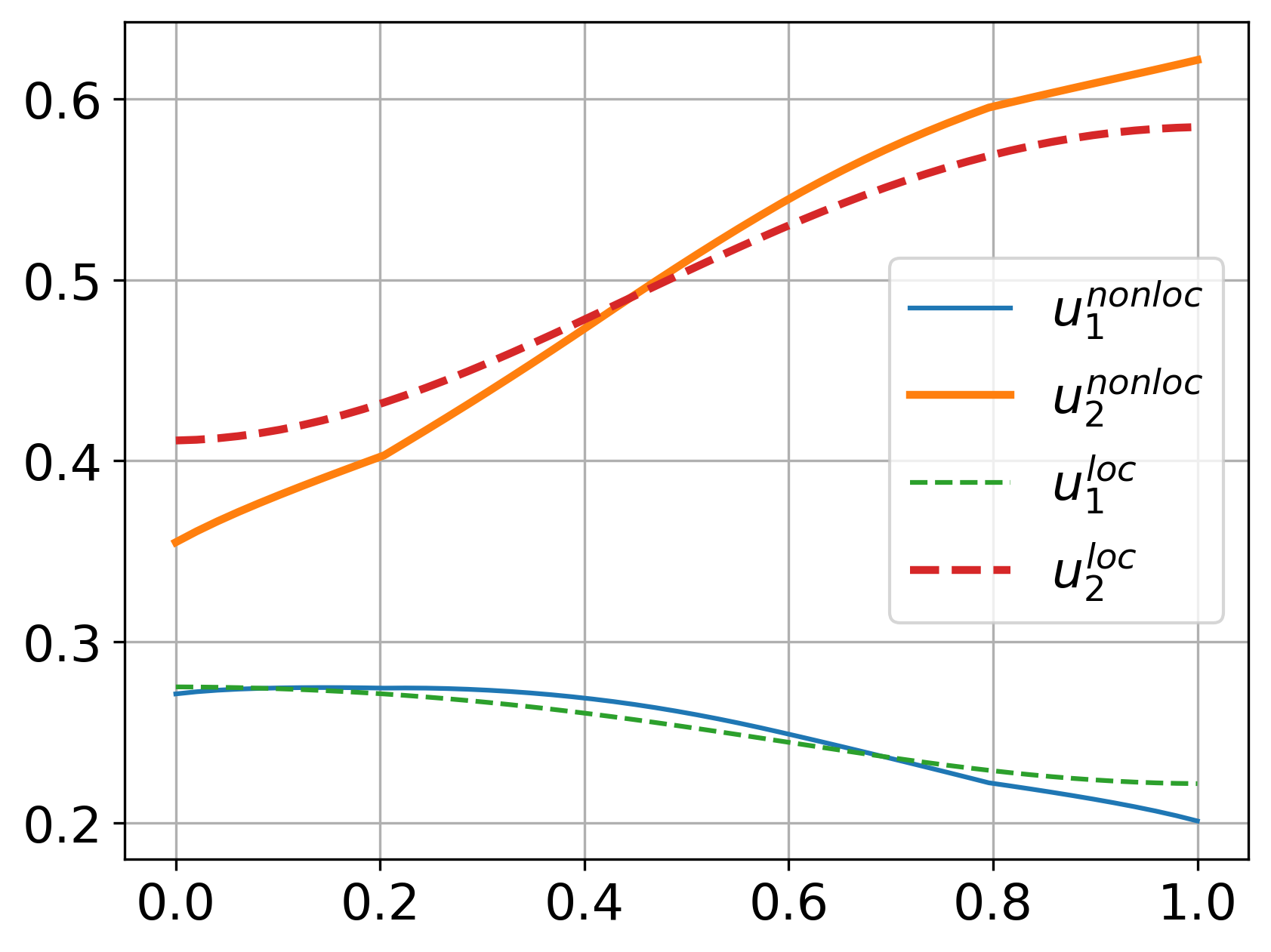}
\includegraphics[width=65mm]{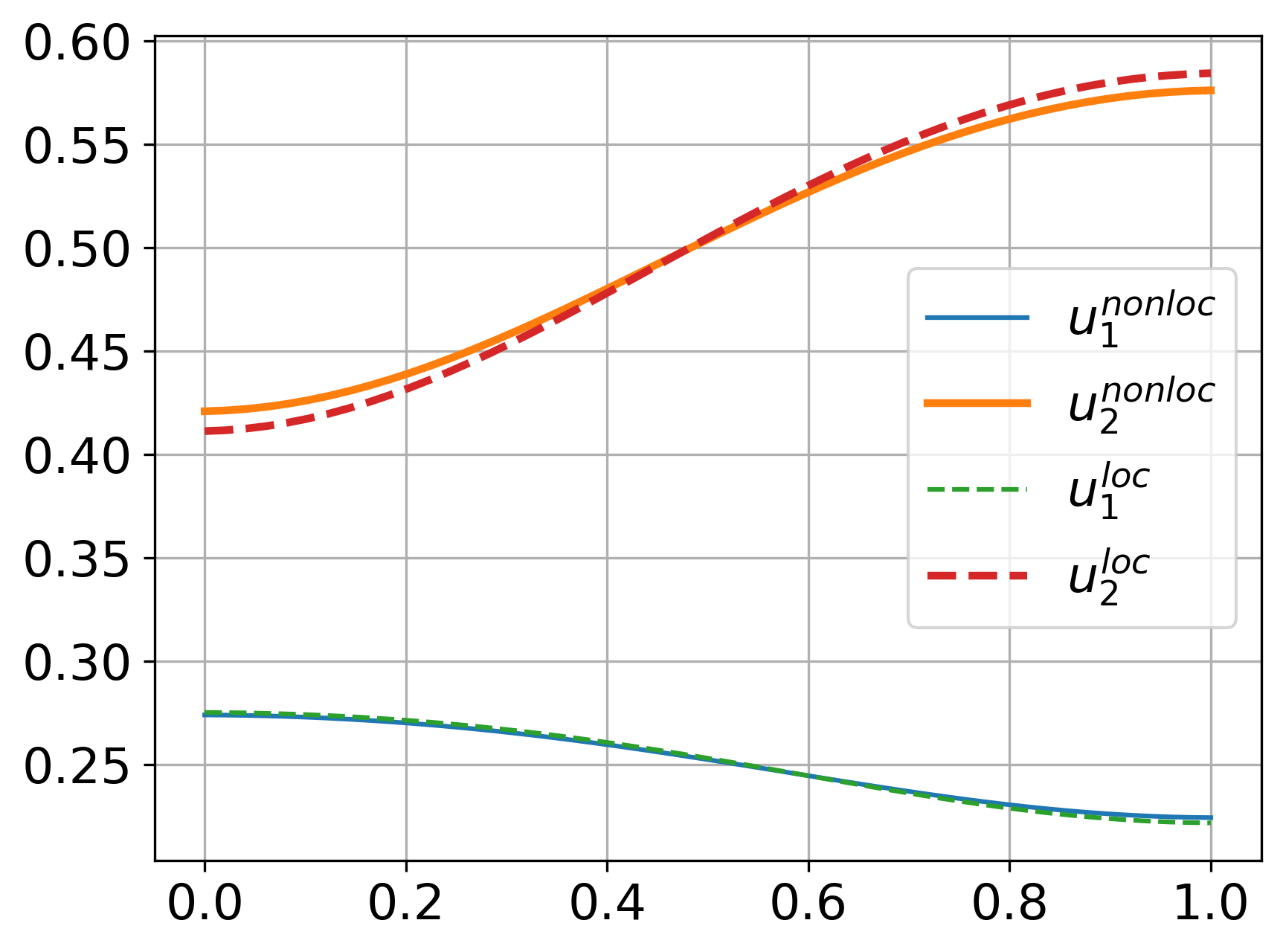}
\includegraphics[width=65mm]{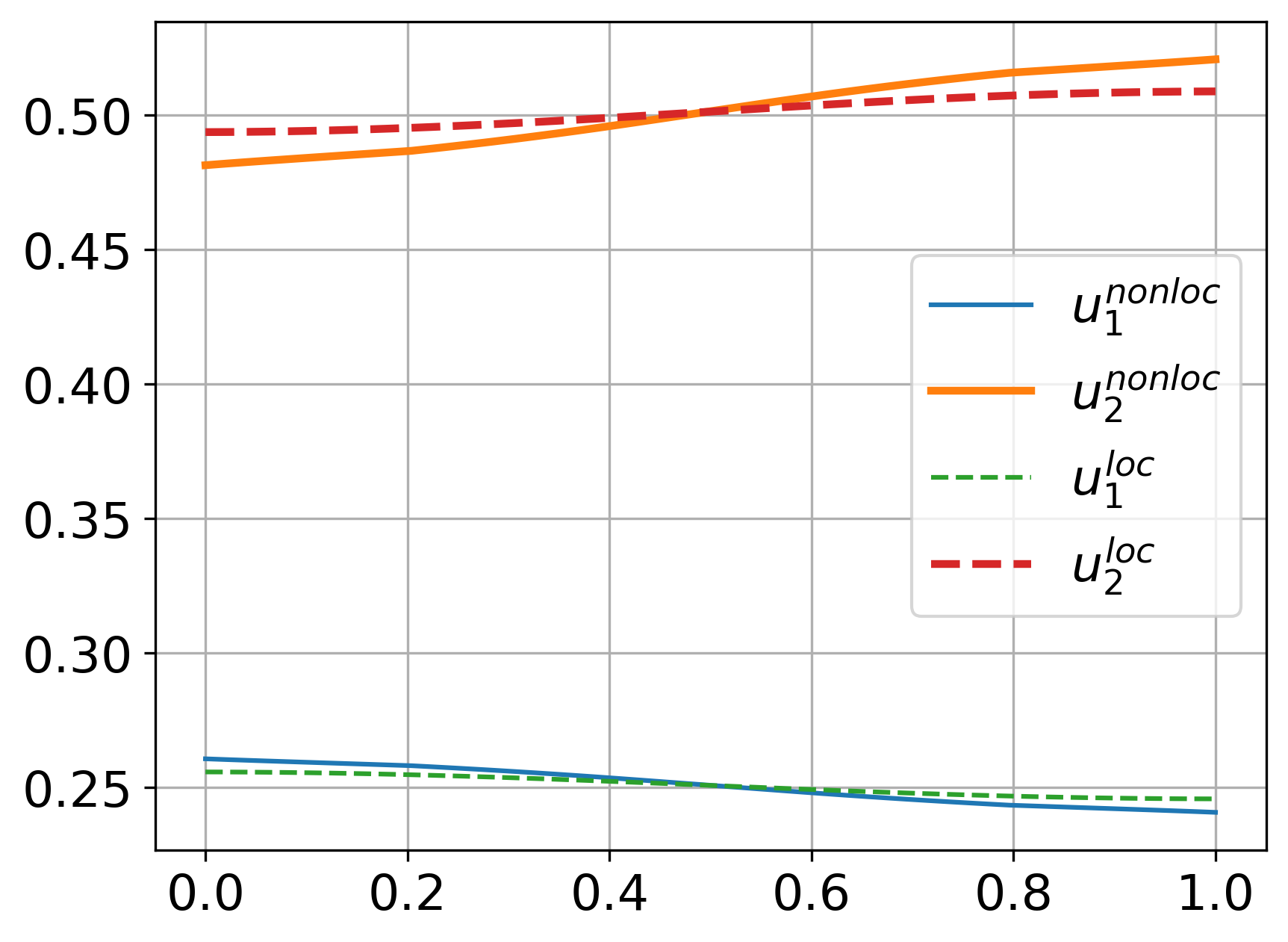}
\includegraphics[width=65mm]{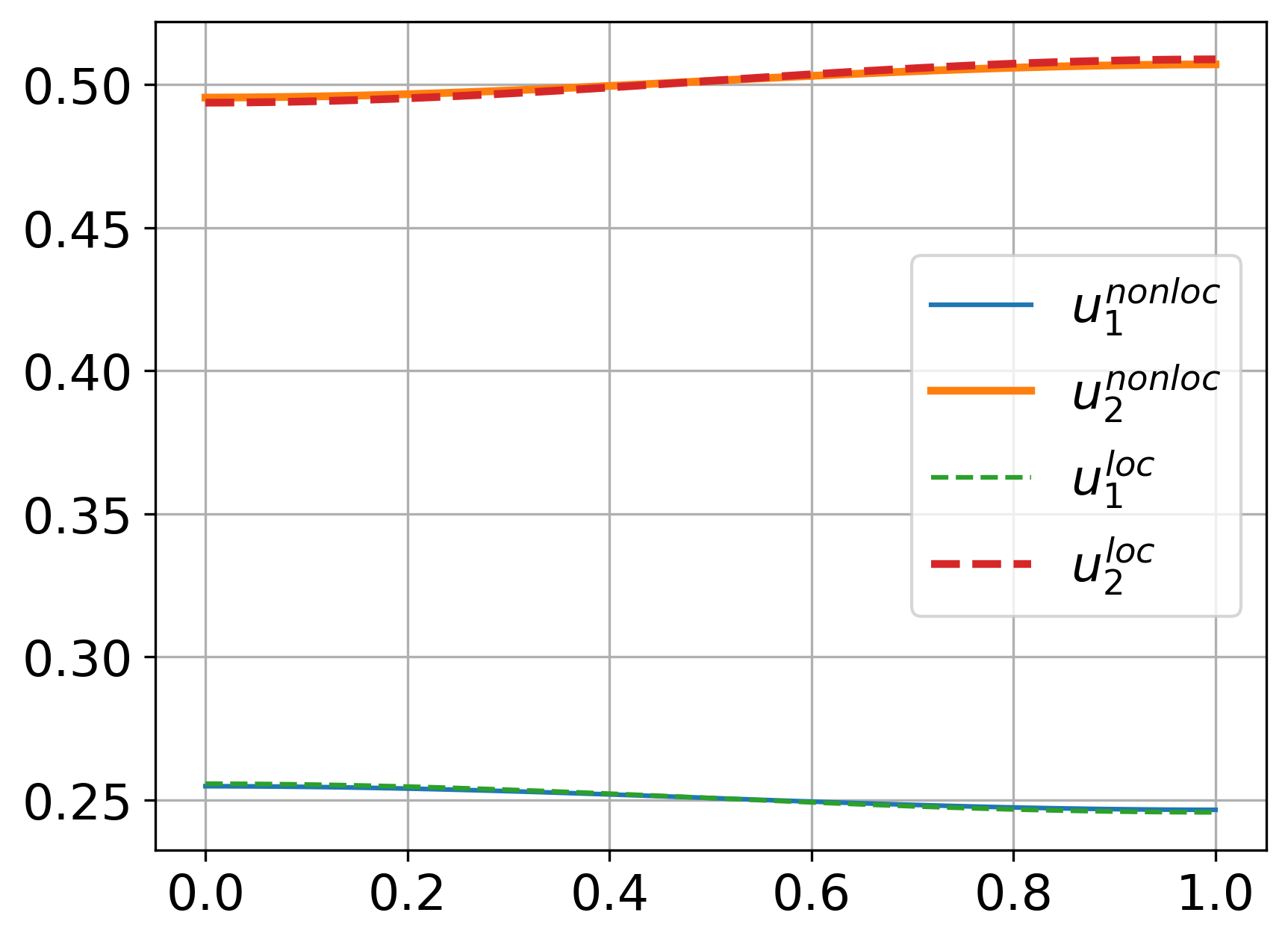}
\caption{Solutions $u_1$, $u_2$ to the nonlocal and local model for $\eps=0.4$ (left column) and $\eps= 0.05$ (right column) at times $t=50$ (top row)  and $t=600$ (middle row), and $t=1500$ (bottom row).}
\label{fig.loc}
\end{figure}

Next, we fix $\eps=0.08$. Figure \ref{fig.eps} illustrates the behavior of the solutions for different values of the exponent $s$. The simulations suggest that the solutions behave similarly for different values of $s$, indicating that analogous analytical results should also hold for $s < 1$.

\begin{figure}[ht]
\includegraphics[width=65mm]{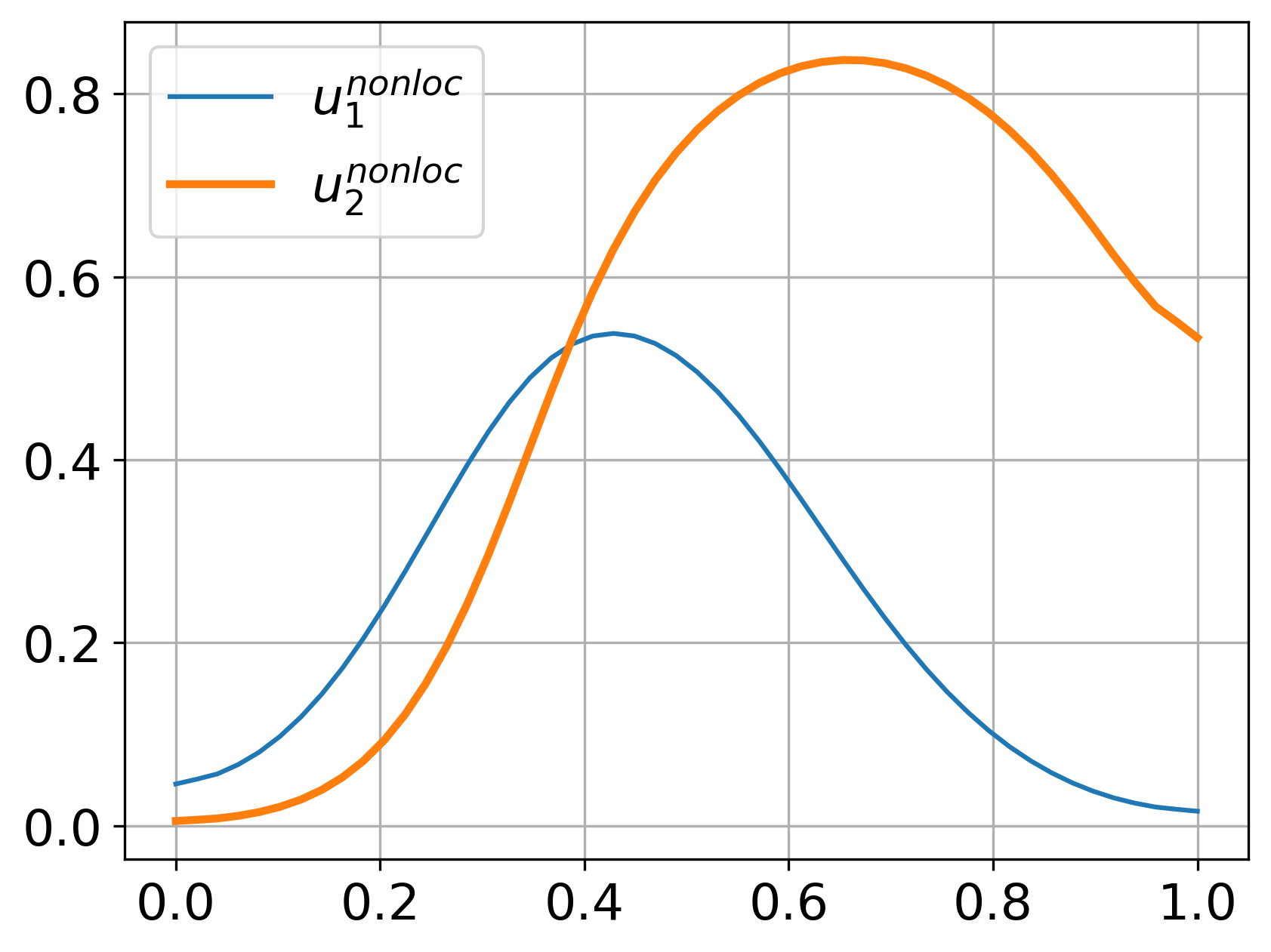}
\includegraphics[width=65mm]{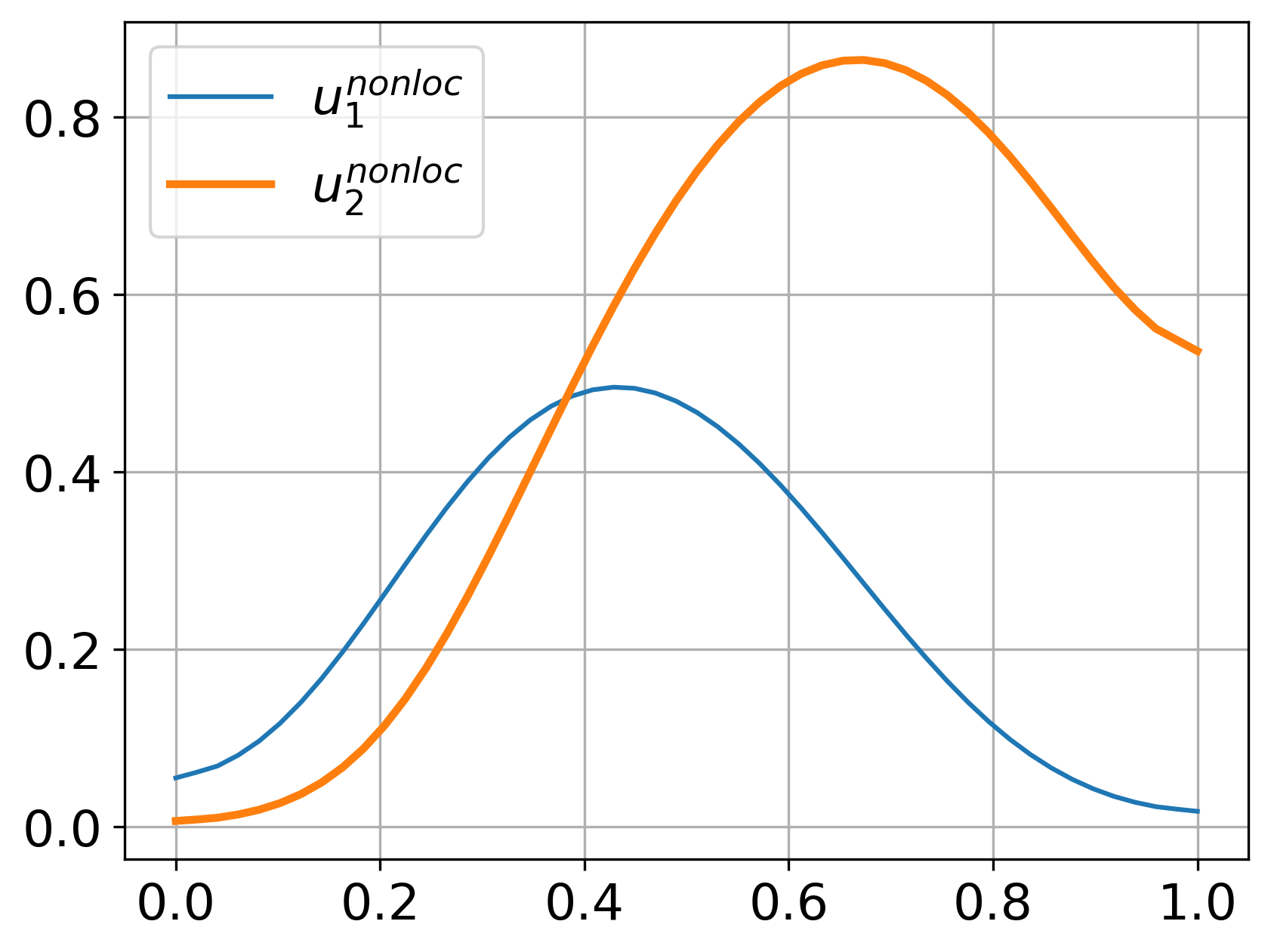}
\includegraphics[width=65mm]{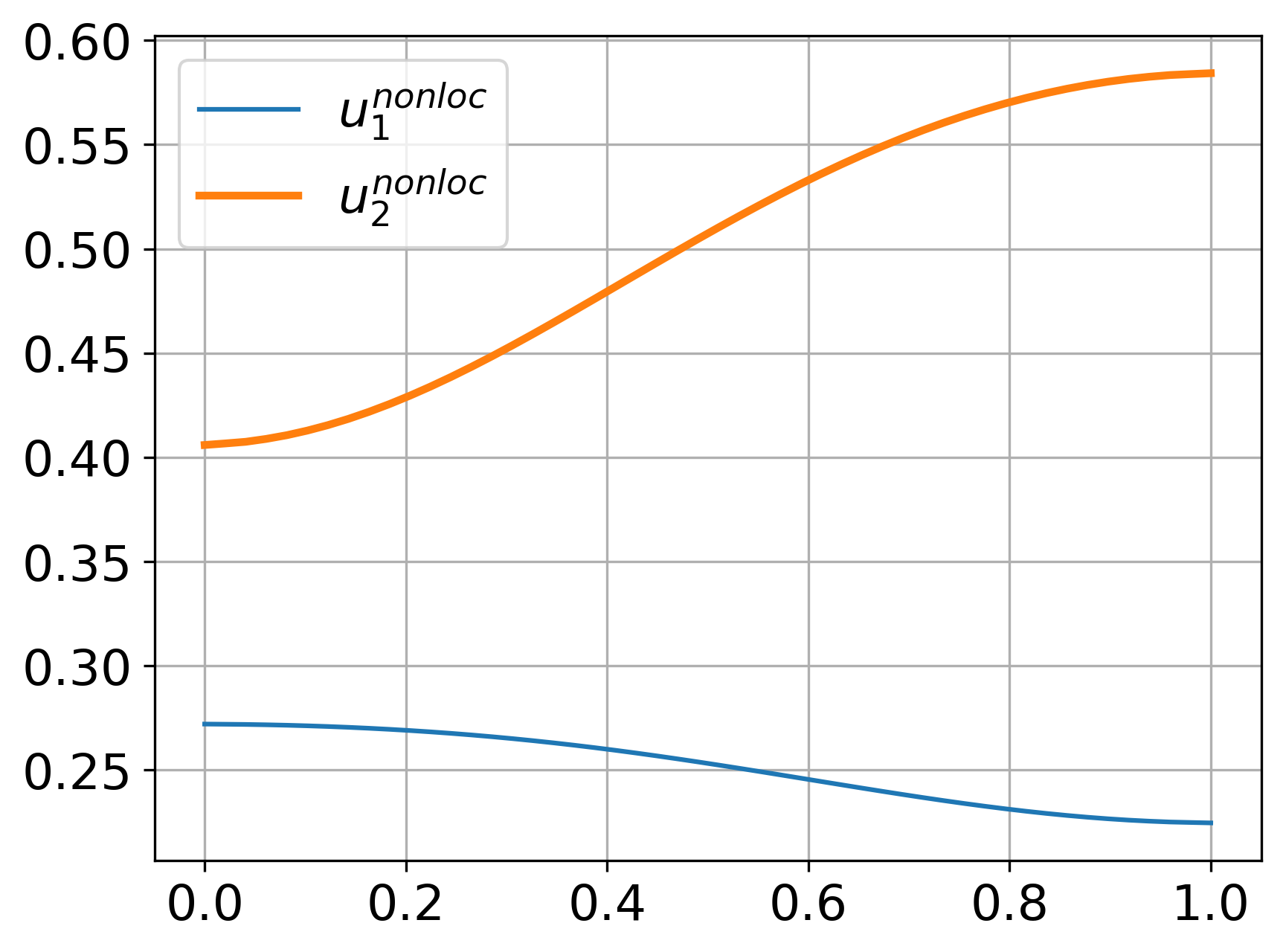}
\includegraphics[width=65mm]{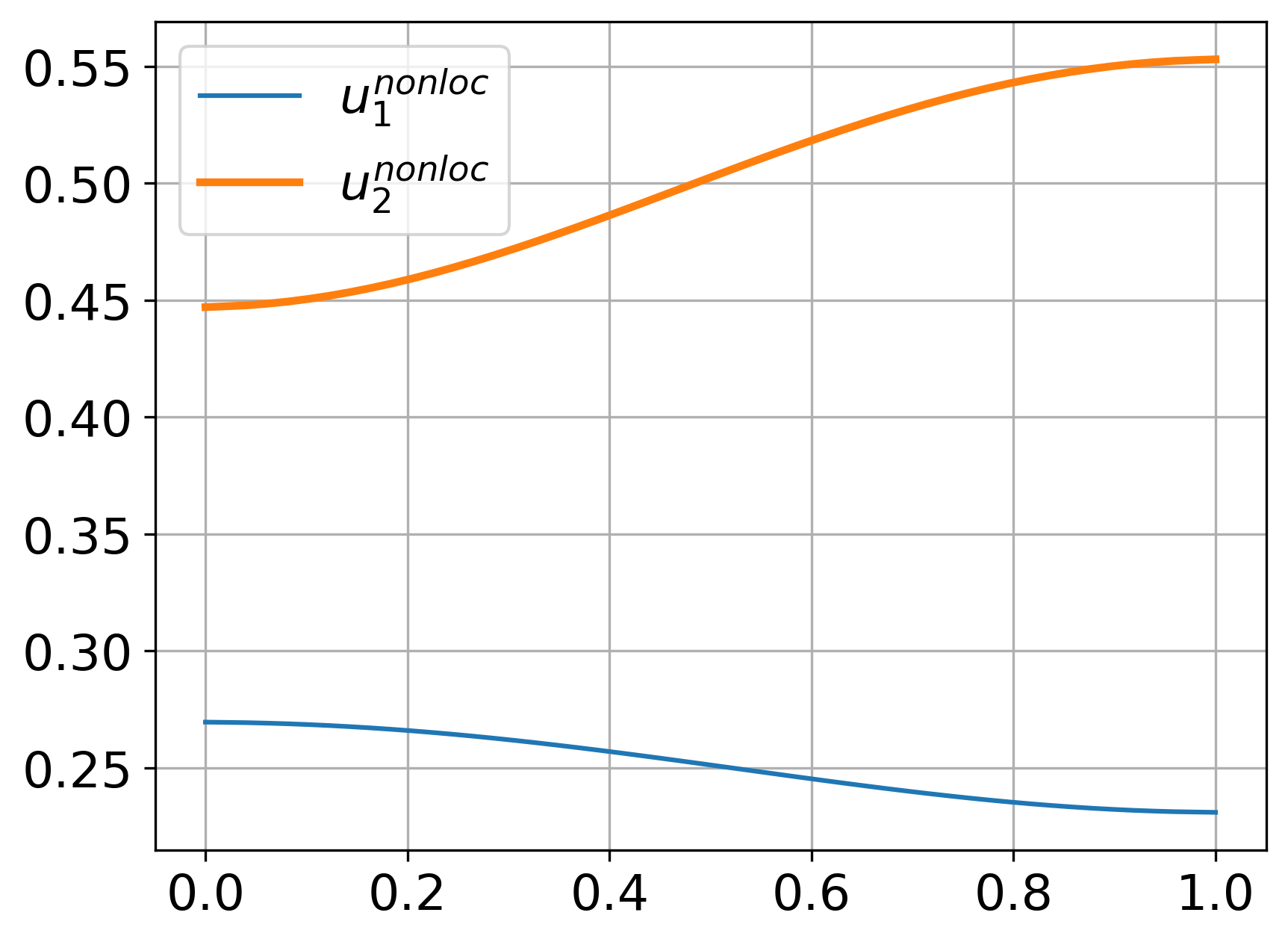}
\caption{Solutions $u_1$, $u_2$ to the nonlocal model for $s=0.5$ (left column) and $s=4$ (right column) at times $t=50$ (top row) and $t=600$ (bottom row).}
\label{fig.eps}
\end{figure}


\begin{appendix}
\section{Entropy-dissipating finite-volume schemes}\label{sec.fvm}

We show in this section that the discrete chain rules in Lemmas \ref{lem.chain1}--\ref{lem.chain2} can be used to construct structure-preserving finite-volume schemes for cross-diffusion systems of the type $\pa_t u_i = \Delta P_i(u) + u_if_i(u)$ with initial and no-flux boundary conditions. This extends the results in \cite{JuZu21}. 

We discretize the domain $\Omega$ by an admissible triangulation in the sense of \cite[Def.~9.1]{EGH00}. The triangulation consist of a family $\mathcal{T}$ of open polygonal convex subsets of $\Omega$ (cells or control volumes), a family $\mathcal{E}$ of edges or faces, and a family $\mathcal{P}=(x_K)_{K\in\mathcal{T}}$ of points associated to the cells. The triangulation is chosen in such a way that the line segment connecting the points $x_K$ and $x_L$ of two neighboring cells is
orthogonal to their common edge $\sigma = K|L$. The family of edges $\mathcal{E}$ is split into internal and external edges $\mathcal{E}=\mathcal{E}_{\rm int}\cup\mathcal{E}_{\rm ext}$, where $\mathcal{E}_{\rm int} = \{\sigma\in\mathcal{E}:\sigma\subset\Omega\}$ and $\mathcal{E}_{\rm ext} = \{\sigma\in\mathcal{E}: \sigma\subset\pa\Omega\}$. The set $\mathcal{E}_K$ contains all edges associated to the cell $K$. For given $\sigma\in\mathcal{E}$, we define
\begin{align*}
  \rm{d}_\sigma = \begin{cases}
  d(x_K,x_L) &\mbox{if }\sigma=K|L\in\mathcal{E}_{\rm int}, \\
  d(x_K,\sigma) &\mbox{if }\sigma\in\mathcal{E}_{\rm ext},
  \end{cases}
\end{align*}
where $d(x,y)$ is the Euclidean distance between two points $x$ and $y$. The transmissibility coefficient is defined by $\tau_\sigma= \mathrm{m}(\sigma)/\rm{d}_\sigma$, where $\mathrm{m}(\sigma)$ denotes the Lebesgue measure of $\sigma$. 

We use a uniform time discretization with time step $\Delta t>0$, and we set $t_k=k\Delta t$ for $k=1,\ldots,N$, where $T>0$, $N\in\N$, and $\Delta t=T/N$. Finally, we define for $v=(v_K)_{K\in\mathcal{T}}$ and $\sigma\in\mathcal{E}$ the difference
\begin{align*}
  \mathrm{D}_{K,\sigma}v = v_{K,\sigma}-v_K, \quad\mbox{where}\quad
  v_{K,\sigma} = \begin{cases}
  v_L &\mbox{if }\sigma=K|L\in\mathcal{E}_{{\rm int},K}, \\
  v_K &\mbox{if }\sigma\in\mathcal{E}_{{\rm ext},K}.
  \end{cases}
\end{align*}

The implicit Euler finite-volume scheme reads as
\begin{align}\label{4.eq}
  \mathrm{m}(K)\frac{u_{i,K}^k-u_{i,K}^{k-1}}{\Delta t}
  - \sum_{\sigma\in\mathcal{E}_K}\tau_\sigma
  \mathrm{D}_{K,\sigma}P_i(u_K^k)
  &= \mathrm{m}(K)u_{i,K}^kf_i(u_K^k).
\end{align}
The advantage of the present scheme over that proposed in \cite{JuZu21} is that it neither requires the introduction of edge values nor the evaluation of a Stolarsky-type mean function, which may be computationally delicate. By construction, scheme \eqref{4.eq} preserves the local flux and the total mass. The important point is that this scheme also preserves the entropy structure in the following sense.

\begin{proposition}[Discrete entropy inequality]
Let Assumptions (A5a) or (A5b) hold and $(u_K^k)_{K\in\mathcal{T}}$ be a solution to \eqref{4.eq}. Then there exist constants $C_0$, $C_1>0$ independent of the size of the cells and the time step size $\Delta t$ such that 
\begin{align*}
  \sum_{K\in\mathcal{K}}&\mathrm{m}(K)h(u_K^k)
  - \sum_{K\in\mathcal{K}}\mathrm{m}(K)h(u_K^{k-1})
  + C_0\Delta t\sum_{\sigma\in\mathcal{E}}\tau_\sigma
  \big((u_{i,L}^k)^{s/2}-(u_{i,K}^k)^{s/2}\big)^2 \\
  &+ C_0\Delta t\sum_{i=1}^n\sum_{\sigma\in\mathcal{E}}\tau_\sigma
  \big((u_{i,L}^k)^{s}-(u_{i,K}^k)^{s}\big)^2
  \le C_1 + C_1\Delta t\sum_{K\in\mathcal{T}}\mathrm{m}(K)h(u_K^k),
\end{align*}
which is the analog of the nonlocal entropy inequality \eqref{1.ei}. 
\end{proposition}

\begin{proof}
Assume that Assumption (A5a) holds. We multiply the scheme by the function $h'_i(u_{i,K}^k) = (s (u_{i,K}^k)^{s-1}-1)/(s-1)$ and sum over $K\in\mathcal{T}$ and $i=1,\ldots,n$, leading to $J_0+J_1+J_2=0$, where
\begin{align*}
  J_0 &= \sum_{i=1}^n\sum_{K\in\mathcal{K}}\mathrm{m}(K)
  \frac{u_{i,K}^k-u_{i,K}^{k-1}}{\Delta t}h'_i(u_{i,K}^k), \\
  J_1 &= -\sum_{i=1}^n\sum_{K\in\mathcal{K}}
  \sum_{\sigma\in\mathcal{E}_K}\tau_\sigma
  \mathrm{D}_{K,\sigma}P_i(u^k)h'_i(u_{i,K}^k), \\
  J_2 &= -\sum_{i=1}^n\sum_{K\in\mathcal{K}}\mathrm{m}(K)
  u_{i,K}^kf_i(u_K^k)h'_i(u_{i,K}^k).
\end{align*}
The convexity of $h_i$ implies that
\begin{align*}
  J_0\ge \frac{1}{\Delta t}\sum_{i=1}^n
  \sum_{K\in\mathcal{K}}\mathrm{m}(K)
  \big(h_i(u_{i,K}^{k}) - h_i(u_{i,K}^{k-1})\big)
  = \frac{1}{\Delta t}\sum_{K\in\mathcal{K}}\mathrm{m}(K)
  \big(h(u_{K}^{k}) - h(u_{K}^{k-1})\big).
\end{align*}
By discrete integration by parts, 
\begin{align*}
  J_1 &= \sum_{i=1}^n\sum_{\sigma\in\mathcal{E}}\tau_\sigma
  \mathrm{D}_{K,\sigma}P_i(u^k)\mathrm{D}_{K,\sigma}h'_i(u_{i}^k) \\
  &= \sum_{i=1}^n\sum_{\sigma=K|L\in\mathcal{E}_{\rm int}}\tau_\sigma
  \big(P_i(u_L^k) - P_i(u_K^k)\big)
  \big(h'_i(u_{i,L}^k) - h'_i(u_{i,K}^k)\big)
\end{align*}
This expression resembles the term $I_1$ in \eqref{3.I12} with $u_{i,L}^k = X_i$ and $u_{i,K}^k = Y_i$ in the proof of Lemma \ref{lem.aei1}, and repeating the arguments therein, we find that
\begin{align*}
  J_1 &\ge \frac{2}{s}\sum_{i=1}^n a_{i0}
  \sum_{\sigma\in\mathcal{E}}\tau_\sigma
  \big((u_{i,L}^k)^{s/2}-(u_{i,K}^k)^{s/2}\big)^2 \\
  &\phantom{xx}+ \frac{s}{8(s-1)C(s)}\sum_{i=1}^n
  \bigg(a_{ii}-C(s)\sum_{j=1, j\neq i}^n a_{ij}\bigg)
  \sum_{\sigma\in\mathcal{E}}\tau_\sigma
  \big((u_{i,L}^k)^{s}-(u_{i,K}^k)^{s}\big)^2.
\end{align*}
The term $J_2$ resembles $I_2$ in \eqref{3.I12} and can be estimated in the same way, giving
\begin{align*}
  J_2 \ge -C - C\sum_{K\in\mathcal{T}}\mathrm{m}(K)h(u_K^k).
\end{align*}
Under Assumption (A5b), the arguments are similar; see the proof of Lemma \ref{lem.aei2}. Collecting the estimates finishes the proof.
\end{proof}
\end{appendix}



\begin{thebibliography}{11}

\bibitem{AFP00} L.~Ambrosio, N.~Fusco, and D.~Pallara. {\em Functions of Bounded Variation and Free Discontinuity Problems}. Oxford University Press, New York, 2000.

\bibitem{AMRT10} F.~Andreu-Vaillo, J.-M.~Maz\'on, J.~Rossi, and J.~Toledo-Melero. {\em Nonlocal Diffusion Problems}. Amer. Math. Soc., Providence, 2010.

\bibitem{BuEs23} M.~Burger and A.~Esposito. Porous medium equation and cross-diffusion systems as limit of nonlocal interaction. {\em Nonlin. Anal.} 235 (2023), no.~113347, 30 pages. 

\bibitem{CHS18} J.~A.~Carrillo, Y.~Huang, and M.~Schmidtchen. Zoology of a nonlocal cross-diffusion model for two species. {\em SIAM J. Appl. Math.} 78 (2018), 1078--1104. 

\bibitem{CCR06} E.~Chasseigne, M.~Chaves, and J.~Rossi. Asymptotic behavior for nonlocal diffusion equations. {\em J. Math. Pures Appl.} 86 (2006), 271--291.

\bibitem{CJS16} C.~Chainais-Hillairait, A.~J\"ungel, and S.~Schuchnigg. Entropy-dissipative discretization of nonlinear diffusion equations and discrete Beckner inequalities. {\em ESAIM: Math. Model. Anal. Numer.} 50 (2016), 135--162.

\bibitem{CDHJ21} L.~Chen, E.~Daus, A.~Holzinger, and A.~J\"ungel. Rigorous derivation of population cross-diffusion systems from moderately interacting particle systems. {\em J. Nonlin. Sci.} 31 (2021), no.~94, 38 pages.

\bibitem{CDJ18} X.~Chen, E.~Daus, and A.~J\"ungel. Global existence analysis of cross-diffusion population systems for multiple species. {\em Arch. Ration. Mech. Anal.} 227 (2018), 715--747.

\bibitem{ChJu19} X.~Chen and A.~J\"ungel. Weak--strong uniqueness of renormalized solutions to reaction--cross-diffusion systems. {\em Math. Models Meth. Appl. Sci.} 29 (2019), 237--270.

\bibitem{Cla87} D.~Clark. Short proof of a discrete Gronwall inequality. {\em Discrete Appl. Math.} 16 (1987), 279--281.

\bibitem{DHPP24} M.~Doumic, and S.~Hecht, B.~Perthame, and D.~Peurichard. Multispecies cross-diffusions: From a nonlocal mean-field to a porous medium system without self-diffusion. {\em J. Differ. Eqs.}
389 (2024), 228--256.

\bibitem{DiMo24} H.~Dietert and A.~Moussa. Persisting entropy structure for nonlocal cross-diffusion systems. {\em Ann. Fac. Sci. Toulouse} 33 (2024), 69--104.

\bibitem{EGH00} R.~Eymard, T.~Gallou\"et, and R.~Herbin. Finite volume methods. In: P.~G.~Ciarlet and J.-L.~Lions (eds.), {\em Handbook of Numerical Analysis} 7 (2000), 713--1018.

\bibitem{DEF18} M.~Di Francesco, A.~Esposito, and S.~Fagioli. Nonlinear degenerate cross-diffusion systems with nonlocal interaction. {\em Nonlin. Anal.} 169 (2018), 94--117.

\bibitem{GaVe19} G.~Galiano and J.~Velasco. Well-posedness of a cross-diffusion population model with nonlocal diffusion. {\em SIAM J. Math. Anal.} 51 (2019), 2884--2902.

\bibitem{GaVe22} G.~Galiano and J.~Velasco. Convergence of solutions of a rescaled evolution nonlocal cross-diffusion problem to its local diffusion counterpart. {\em Rev. Real Acad. Cienc. Exactas Fis. Nat. Ser. A} 116 (2022), no.~93, 17 pages.

\bibitem{JPZ22} A.~J\"ungel, S.~Portisch, and A.~Zurek. Nonlocal cross-diffusion systems for multi-species populations and networks. {\em Nonlin. Anal.} 219 (2022), no.~112800, 26 pages.

\bibitem{JuZa16} A.~J\"ungel and N.~Zamponi. Qualitative behavior of solutions to cross-diffusion systems from population dynamics. {\em J. Math. Anal. Appl.} 440 (2016), 794--809.

\bibitem{JuZu21} A.~J\"ungel and A.~Zurek. A convergent structure-preserving finite-volume scheme for the Shigesada--Kawasaki--Teramoto population system. {\em SIAM J. Numer. Anal.} 59 (2021), 2286--2309.

\bibitem{Lin90} P.~Lindqvist. On the equation $\diver(|\na u|^{p-2}\na u)+\lambda|u|^{p-2}u=0$. {\em Proc. Amer. Math. Soc.} 109 (1990), 157--164.

\bibitem{Mou20} A.~Moussa. From nonlocal to classical Shigesada--Kawasaki--Teramoto systems: triangular case with
bounded coefficients. {\em SIAM J. Math. Anal.} 52 (2020), 42-–64.

\bibitem{SKT79} N.~Shigesada, K.~Kawasaki, and E.~Teramoto. Spatial segregation of interacting species. {\em J. Theor. Biol.} 79 (1979), 83--99.

\bibitem{Tsa88} C.~Tsallis. Possible generalization of Boltzmann–Gibbs statistics. {\em J. Stat. Phys.} 52 (1988), 479--487.

\end{thebibliography}
\end{document}